\documentclass[11pt]{article}

\usepackage[margin=1.25in]{geometry}

\usepackage{amsmath,amsthm,amssymb,tikz,bbm,float}
\usepackage{lineno}

\newtheorem{theorem}{Theorem}[section]
\newtheorem{corollary}[theorem]{Corollary}
\newtheorem{lemma}[theorem]{Lemma}
\newtheorem{proposition}[theorem]{Proposition}
\newtheorem{conjecture}[theorem]{Conjecture}
\theoremstyle{definition}
\newtheorem{definition}[theorem]{Definition}

\newtheorem{example}[theorem]{Example}
\newtheorem{note}[theorem]{Note}

\title{Categorification of integral group rings \\ extended by one dimension}

\author{Andrew Schopieray}

\date{}

\begin{document}

\maketitle
%
%\affiliation[label1]{organization={Department of Mathematical and Statistical Sciences,  University of Alberta},%Department and Organization
%            addressline={CAB 632}, 
%            city={Edmonton},
%            postcode={T6G\,2G1}, 
%            state={Alberta},
%            country={Canada}}

\begin{abstract}
The integral group rings $\mathbb{Z}G$ for finite groups $G$ are precisely those fusion rings whose basis elements have Frobenius-Perron dimension 1, and each is categorifiable in the sense that it arises as the Grothendieck ring of a fusion category.  Here we analyze the structure and representation theory of fusion rings with a basis of elements whose Frobenius-Perron dimensions take exactly one value distinct from 1.  Our goal is a set of results which assist in characterizing when such fusion rings are categorifiable.  As proof of concept, we completely classify the categorifiable near-group fusion rings for an infinite collection of finite abelian groups, a task that to-date has only been completed for three such groups.
\end{abstract}

%%%%%%%%%%%%%%%%%%%%%

%%%%%%%%%%%%%%%%%%%%%

%%%%%%%%%%%%%%%%%%%%%

\section{Introduction}\label{sec:intro}

\par Fusion rings capture the combinatorial essence of many categorical constructions which appear in representation theory.  Their ilk include $C$-algebras, reality-based algebras, table algebras, and association schemes,\ to name a few.  For a survey and comparison of some of these notions one can refer to \cite{MR2535395}.  The canonical example of a fusion ring is the integral group ring $\mathbb{Z}G$ of a finite group $G$ whose basis consists of invertible elements, i.e.\ those which possess a multiplicative inverse, which occurs if and only if the element has Frobenius-Perron dimension 1.  These vast infinite families of fusion rings are \emph{categorifiable} inasmuch as they can be realized as the Grothendieck rings of fusion categories: the categories of finite-dimensional $G$-graded complex vector spaces in this case.  For a given fusion ring, whether it possesses a categorification can be an incredibly difficult question to answer even though it is known that for each fusion ring, there are only finitely many categorifications up to equivalence \cite[Theorem 9.1.4]{tcat}.  This categorification question is open even for the simplest generalization of $\mathbb{Z}G$, which is to assume there exists a unique basis element $\rho$ which is not invertible, and therefore the remainder of the basis elements form a finite group $G$ under multiplication.  These \emph{near-group} fusion rings are indexed by the group $G$ and a nonnegative integer level $\ell$ such that $\rho^2=\ell\rho+\sum_{g\in G}g$.  For this reason we denote the near-group fusion rings by $R(G,\ell)$ in what follows.  With the demand that the categorification be unitary, which is a nontrivial assumption \cite{MR4327964}, it is known that the only $R(G,\ell)$ which are categorifiable for nonabelian $G$ are the extra-special $2$-groups of order $2^{2m+1}$ for $m\in\mathbb{Z}_{\geq1}$ with $\ell=2^m$ \cite[Theorem 6.1]{MR3635673}.  When $G$ is abelian, $R(G,0)$ is always categorifiable \cite{MR1659954} and so the question is more subtle: for which $\ell\in\mathbb{Z}_{\geq0}$ is $R(G,\ell)$ categorifiable?  Victor Ostrik proved in the context of classifying low-rank fusion categories that $R(C_1,\ell)$ is categorifiable if and only if $\ell\in\{0,1\}$ \cite{ostrik} and $R(C_2,\ell)$ is categorifiable if and only if $\ell\in\{0,1,2\}$ \cite{ost15} where $C_n$ is the cyclic group of order $n\in\mathbb{Z}_{\geq1}$.  Hannah Larson later proved that $R(C_3,\ell)$ is categorifiable if and only if $\ell\in\{0,2,3,6\}$ \cite[Section 4]{MR3229513}.  Over the 20 years since the term ``near-group categories'' was coined \cite{MR1997336}, the question of when $R(G,\ell)$ is categorifiable with $G$ abelian has been completely answered only for these three finite abelian groups.  As an application of our general theory outlined below, we extend these results to the elementary abelian $2$-groups, i.e.\ $G=C_2^m$ for $m\in\mathbb{Z}_{\geq1}$.

\begin{theorem}\label{theend}
Let $G$ be an elementary abelian $2$-group and $\ell\in\mathbb{Z}_{\geq0}$.  Then the near-group fusion ring $R(G,\ell)$ is categorifiable if and only if
\begin{itemize}
\item[(a)] $\ell=0$,
\item[(b)] $G=C_2$ and $\ell=1$ or $\ell=2$, or
\item[(c)] $G=C_2^2$ and $\ell=4$.
\end{itemize}
\end{theorem} 
\par There are two very different tasks to complete to prove Theorem \ref{theend} and similar classification results.  On one hand, for a fixed finite abelian group $G$, one must construct fusion categories which categorify $R(G,\ell)$ for a finite number of levels $k$.  Daisuke Tambara \& Shigeru Yamagami used a brute-force approach to demonstrate that $R(G,0)$ is categorifiable for all finite abelian groups $G$; these methods do not generalize in an obvious way to $\ell>0$.  There is much literature devoted to constructing categorifications of near-group fusion rings using operator algebras.  Key expositions on this topic are due to Masaki Izumi \cite{MR1782145,MR1832764,MR3635673}, and David Evans \& Terry Gannon \cite{MR2837122,MR3167494}.  One can find proof that $R(C_2,2)$ and $R(C_2^2,4)$ are categorifiable in \cite{MR1832764}; the fusion ring $R(C_2,1)$ is categorified by $\mathrm{Rep}(S_3)$, the category of finite-dimensional complex representations of the symmetric group $S_3$.  It has been conjectured that $R(C_\ell,\ell)$ is categorifiable for all $\ell\in\mathbb{Z}_{\geq1}$ \cite[Conjecture 1]{MR3167494}.   Together with \cite[Proposition 6]{MR3167494}, Theorem \ref{theend} proves that $G=C_2^3$ is the abelian group of smallest order such that $R(G,|G|)$ is not categorifiable.

\par On the other hand, one must also demonstrate that for a fixed finite abelian group $G$, there exists an upper bound on the levels $\ell\in\mathbb{Z}_{\geq0}$ such that $R(G,\ell)$ is categorifiable.  To do so, we study the \emph{double} (or Drinfeld center \cite[Section 7.13]{tcat}) of the hypothetical fusion categories.  These doubles are modular tensor categories \cite[Section 8.14]{tcat}, partially controlled by ($\dagger$) the representation theory of the fusion ring, and by ($\ddagger$) numerical constraints arising from the representation theory of the modular group $\mathrm{SL}(2,\mathbb{Z})$.  For this reason, the main body of this paper is devoted to developing a better understanding of ($\dagger$) and ($\ddagger$) for more general fusion rings whose basis elements have one of two distinct Frobenius-Perron dimensions.  We casually refer to these as \emph{two-dimension} fusion rings.

\par In Section \ref{sec:fus}, we introduce the basic notions in the study of fusion rings including dimensions and formal codegrees, and we prove elementary results about the structure of two-dimension fusion rings.  When the finite group of invertible basis elements acts transitively on the noninvertible basis elements, e.g.\ when the nontrivial dimension is irrational, the representation theory is somewhat restrictive.  We casually refer to these as \emph{two-orbit} fusion rings, which have previously been referred to as generalized near-group fusion rings in a categorical context (e.g.\  \cite{thornton2012generalized}).  In Section \ref{twoorbit}, we characterize the irreducible representations of these two-orbit fusion rings vanishing on the nontrivial basis elements (Proposition \ref{triv}).  When the fusion rules between the noninvertible basis elements are \emph{uniform}, in the sense of Section \ref{sec:uni}, the representation theory is completely determined (Proposition \ref{reps}) if the invertible elements commute and the noninvertible basis elements are self-dual.   A satisfying consequence of this characterization is that if $R$ is a fusion ring of this type whose invertible objects (forming a finite group $G$) act fixed-point free on the noninvertible basis elements, then the irreducible representations of $R$ correspond to the irreducible representations of $G\rtimes_\theta C_2$ for a distinguished $\theta\in\mathrm{Aut}(G)$.  This provides a generalization of the observed relationship \cite[Section 2]{MR2837122} between the cyclic \emph{Haagerup-Izumi} fusion rings and the dihedral groups $D_n$ for $n\in\mathbb{Z}_{\geq3}$ on a representation-theoretic level.

\par In Section \ref{sec:cat}, we introduce the basic notions in the study of fusion categories, most notably the induction functor $I:\mathcal{C}\to\mathcal{Z}(\mathcal{C})$ from a fusion category $\mathcal{C}$ to its double.  Using the induction functor along with the representation-theoretic results of Section \ref{sec:fus}, we prove several strong divisibility constraints for the categorification of two-dimension and two-orbit fusion rings, and a result on the Galois action of $\mathcal{Z}(\mathcal{C})$ for these categories when $\mathrm{FPdim}(\mathcal{C})$ is irrational.  In Section \ref{sec:twist}, we use Frobenius-Schur indicators \cite{MR2381536} to prove results bounding the twists of simple summands of $I(X)$ for invertible objects $X$ in arbitrary pivotal fusion categories.  Propositions \ref{thelemma}--\ref{theprop} strongly constrain the categorifications of fusion rings possessing many invertible elements, such as two-dimension fusion rings.  Fusion categories whose simple objects take only two integer Frobenius-Perron dimensions have been studied previously in \cite[Section 7]{MR2098028} and \cite{MR2968637}, for example.  The literature on near-group fusion categories is more robust, including \cite{budinski2021exotic,MR4167662,MR3167494,MR3635673,schopierayargentina,MR1997336,MR1659954,thornton2012generalized,MR4044867}.

\par Section \ref{sec:near} is the culmination of the above results.  We index the set of simple objects in the doubles of near-group fusion categories in Section \ref{sec:label}, and provide level bounds for categorifiable near-group fusion rings $R(C_2^m,\ell)$ for $m\in\mathbb{Z}_{\geq1}$ in Section \ref{sec:elem}, proving Theorem \ref{theend}.  As an epilogue, we use the same methods to provide strict constraints on levels for which $R(C_p,\ell)$ is categorifiable when $p\equiv3\pmod{4}$ is prime to inspire future research.  Even modest improvements to the number-theoretic results found in \cite[Section 6]{MR3229513} would complete the search for bounds in these cases, and presumably provide a characterization of when $R(C_p^m,\ell)$ is categorifiable for $m\in\mathbb{Z}_{\geq1}$ using identical techniques to the case $p=2$ found in Section \ref{sec:elem}.  We conjecture, and heuristic argument suggests, that for each finite abelian group $A$, there exists $M_A\in\mathbb{Z}_{\geq1}$ such that $R(A^m,\ell)$ is categorifiable if and only if $\ell=0$ for all $m\geq M_A$.

%%%%%%%%%%%%%%%%%%%%%

%%%%%%%%%%%%%%%%%%%%%

%%%%%%%%%%%%%%%%%%%%%

\section{Fusion rings}\label{sec:fus}

\par The following definitions can be found in a standard textbook such as \cite[Chapter 3]{tcat}.  A unital $\mathbb{Z}_{\geq0}$-ring $(R,B)$ is a pairing of an associative unital ring $R$ which is free as a $\mathbb{Z}$-module, and a distinguished basis $B=\{b_i:i\in I\}$ including the unit $b_0:=1_R$ such that for all $i,j\in I$, $b_ib_j=\sum_{k\in I}c_{i,j}^kb_k$ with $c_{i,j}^k\in\mathbb{Z}_{\geq0}$.  The cardinality of $B$ is referred to as the rank of the fusion ring and denoted $\mathrm{rank}(R)$.  The nonnegative integer structure constants $c_{i,j}^k$ for all $i,j,k\in I$ are the fusion rules of $(R,B)$.  Such a ring is based if there exists an involution of $I$, denoted $i\mapsto i^\ast$, such that the induced map defined by
\begin{equation}
a=\sum_{i\in I}a_ib_i\mapsto a^\ast=\sum_{i\in I}a_ib_{i^\ast},\,\,\, a_i\in\mathbb{Z}
\end{equation}
is an anti-involution of $R$, and $c_{i,j}^0=1$ if $i=j^\ast$ and $c_{i,j}^0=0$ otherwise.  For a fixed unital $\mathbb{Z}_{\geq0}$-ring, being based is a property; there is no choice in the involution should it exist.  We say that $x,y\in R$ are dual to one another when $x=y^\ast$ and $x$ is self-dual when $x=x^\ast$.

\begin{definition}
A \emph{fusion ring} is a unital based $\mathbb{Z}_{\geq0}$-ring $(R,B)$ of finite rank.
\end{definition}

\begin{example}\label{group}
Let $G$ be a finite group.  The integral group ring $R=\mathbb{Z}G$ with basis $B=\{b_g:g\in G\}$ is a fusion ring with fusion rules $b_gb_h:=b_{gh}$ for all $g,h\in G$, unit $b_e$, and duality $b_g^\ast=b_{g^{-1}}$ for all $g\in G$.  As is customary, we will use the group elements themselves to represent their corresponding basis elements from this point forward. 
\end{example}

\par Let $(R,B)$ be a fusion ring.  For each $x\in B$, let $N_x$ be the matrix of fusion coefficients $[c_{x,y}^z]_{y,z\in B}$.  These matrices satisfy a property known as transitivity: for all $x,x'\in B$, there exist $y,y'\in B$ such that $c_{x',y}^x$ and $c_{y',x'}^x$ are nonzero.  As a result of transitivity, the Frobenius-Perron theorem \cite[Proposition 3.2.1]{tcat} states that each $N_x$ has a real eigenvalue $\mathrm{FPdim}(x)$ such that $|\alpha|\leq\mathrm{FPdim}(x)$ for all other eigenvalues $\alpha$ of $N_x$ which we refer to as the Frobenius-Perron dimension of $x$.  Extending linearly from $B$ to $R$ yields a ring homomorphism $\mathrm{FPdim}:R\to\mathbb{R}$ which is the unique character of $R$ taking nonnegative values on $B$.  Furthermore, $\mathrm{FPdim}(x)\geq1$ for all $x\in B$ \cite[Proposition 3.3.4--3.3.6]{tcat}.

\par If $(R,B)$ is a fusion ring and $x\in R$ such that $\mathrm{FPdim}(x)=1$, we say that $x$ is \emph{invertible} since in this case $x\in B$ and $xx^\ast=x^\ast x=1_R$.  The $\mathbb{Z}$-linear span of invertible elements of $R$ forms a fusion subring $R_\mathrm{pt}\subset R$.  In particular the invertible elements of $R$ have the structure of a finite group $G_R$ under multiplication, which acts on $B$ by left (or right) multiplication.  We will use standard notation for group actions in what follows and unless otherwise stated, $G_R$ will act on $B$ by left multiplication.  For example, if $G:=G_R$, then $G\cdot x$ (or $x\cdot G$) is the orbit of $x\in B$ under the left (or right) $G$-action, $B/G$ is the set of $G$-orbits of $B$, $G_x\subset G$ is the stabilizer subgroup of $x\in B$, etc. 

\begin{example}
The Frobenius-Perron dimensions in fusion rings need not be integers.  Let $G$ be a finite group and $\ell\in\mathbb{Z}_{\geq0}$.  The near-group fusion ring $R(G,\ell)$ defined in Section \ref{sec:intro} has a unique noninvertible basis element $\rho$, which is necessarily fixed by all invertible objects under multiplication.  The cyclic symmetry of fusion rules \cite[Proposition 3.1.6]{tcat} implies $c^x_{\rho,\rho^\ast}=c_{x^\ast,\rho}^\rho=1$ for all invertible $x$, so the only unknown fusion coefficient is $\ell:=c_{\rho,\rho}^\rho$.  It follows that $\mathrm{FPdim}(\rho)$ is the maximal root of $x^2-\ell x-|G|$.  If $\mathrm{FPdim}(\rho)\in\mathbb{Z}$, then $\ell<|G|$.  Thus $\ell\geq|G|$ implies $\mathrm{FPdim}(\rho)$ is irrational; these facts are proven and generalized in Lemma \ref{lem1} below.
\end{example}

\par The sum $\mathrm{FPdim}(R):=\sum_{x\in B}\mathrm{FPdim}(x)^2$ for a fusion ring $(R,B)$ is one in a family of distinguishing and highly-structured numerical invariants called \emph{formal codegrees} \cite{codegrees}.  In particular, $R\otimes_\mathbb{Z}\mathbb{C}$ is a semisimple $\mathbb{C}$-algebra and therefore the set of irreducible finite-dimensional representations of $R\otimes_\mathbb{Z}\mathbb{C}$, denoted $\mathrm{Irr}(R)$, are as amenable to study as those of the group algebra $\mathbb{C}G$ for a finite group $G$ (refer to \cite[Section 2]{codegrees} and references therein).  For any $\psi\in\mathrm{Irr}(R)$, the distinguished element $\sum_{x\in B}\mathrm{Tr}(x,\psi)x^\ast$ lies in the center of $R\otimes_\mathbb{Z}\mathbb{C}$ and acts by a scalar $f_\psi\in\mathbb{C}$ on $\psi$.  The set of $f_\psi$ over all $\psi\in\mathrm{Irr}(R)$ are known as the formal codegrees of $R$ since the formal codegrees of $\mathbb{Z}G$ are the classical codegrees $|G|/\dim(\psi)$ for $\psi\in\mathrm{Irr}(G)$ of the finite group $G$.  In general, $f_\psi$ and all of its Galois conjugates are greater than or equal to $\dim(\psi)$ for any $\psi\in\mathrm{Irr}(R)$ \cite[Remark 2.12]{ost15}; we offer a different lower bound in the presence of nontrivial invertible elements.

\begin{lemma}\label{bound}
Let $(R,B)$ be a fusion ring.  Then for any $\psi\in\mathrm{Irr}(R)$, $\dim(\psi)f_\psi=|G_R|+\alpha$
for some $\alpha\geq0$.  In particular, $f_\psi\geq|G_R|/\dim(\psi)$.
\end{lemma}

\begin{proof}
\par We compute $\sum_{x\in B}xx^\ast=|G_R|1_R+\sum_{x\in B\setminus G_R}xx^\ast$.  The eigenvalues of the matrix $M$ of multiplication by this element are $\dim(\psi)f_\psi$ where $f_\psi$ is the formal codegree corresponding to $\psi\in\mathrm{Irr}(R)$  \cite[Lemma 2.6]{codegrees}.  With $I$ the appropriately-sized identity matrix, $M=|G_R|\mathrm{I}+N$ where $N$ is a sum of the positive semidefinite matrices of multiplication by $xx^\ast$ over all $x\in B\setminus G_R$, which is positive semidefinite.  As $I$ and $N$ commute, the eigenvalues of $M$ are  of the form $|G_R|+\alpha$ where $\alpha\geq0$ is an eigenvalue of $N$.
\end{proof}

\par As previously noted, if $\mathrm{FPdim}(x)=\mathrm{FPdim}(y)$ for all $x,y\in B$, then $R=\mathbb{Z}G_R$.  The next level of complexity would be the existence of $d\in\mathbb{R}_{>1}$ such that for all $x\in B$, $\mathrm{FPdim}(x)\in\{1,d\}$.  We casually refer to these as \emph{two-dimension} fusion rings.  As a consequence of applying the ring homomorphism $\mathrm{FPdim}$ to the decomposition of $xx^\ast$ for some $x\in B$ with $d:=\mathrm{FPdim}(x)$, there exist $r\in\mathbb{Z}_{\geq0}$ and $s\in\mathbb{Z}_{\geq1}$ such that
\begin{equation}
d^2-rd-s=0.\label{eq:1}
\end{equation}
In what follows, once a two-dimension fusion ring has been specified, we may refer to the positive root of Equation (\ref{eq:1}), the defining relation of the nontrivial Frobenius-Perron dimension, as $d=d_+$ for brevity, and less commonly, the negative root as $d_-$.   The fusion rules when $d$ is rational or irrational have many differences.

\begin{lemma}\label{lem1}
Let $(R,B)$ be a fusion ring and $d\in\mathbb{R}_{>1}$ such that for all $x\in B$, $\mathrm{FPdim}(x)\in\{1,d\}$.  If $d\in\mathbb{Z}$, then $r<s$, hence $r\geq s$ implies $d$ is irrational.
\end{lemma}

\begin{proof}
Equation (\ref{eq:1}) states that $s=d(d-r)$.  In particular $r<d$ and $d\leq s$.  If $d\in\mathbb{Z}$, then $r=d-(d-r)<s$.
\end{proof}

\begin{lemma}\label{transitive}
Let $(R,B)$ be a fusion ring with $d\in\mathbb{R}_{>1}$ such that for all $x\in B$, $\mathrm{FPdim}(x)\in\{1,d\}$.  If $|B/G_R|>2$, then $d\in\mathbb{Z}$.
\end{lemma}

\begin{proof}
The argument follows from the cyclic symmetry of the fusion rules \cite[Proposition 3.1.6]{tcat} and is consequently left-right symmetric.  Let $x,y\in B\setminus G_R$.  If $y\not\in G\cdot x$,  then for all $g\in G_R$,
\begin{equation}\label{fooor}
c_{y,x^\ast}^g=c_{g^\ast,y}^x=0.
\end{equation}
In other words, no invertible elements appear as summands of $yx^\ast$.  Therefore $\mathrm{FPdim}(yx^\ast)=d^2=rd$ for some $r\in\mathbb{Z}$ and moreover $d=r\in\mathbb{Z}$.
\end{proof}

Although elementary, the following examples demonstrate that Lemma \ref{transitive} is quite nontrivial, in the sense that when $d\in\mathbb{Z}$, $B$ can have any number of $G_R$-orbits.

\begin{example}\label{extra}\,
\begin{itemize}
\item[(a)] Let $n\in\mathbb{Z}_{\geq3}$ and $D_n$ be the dihedral group of order $2n$.  When $n$ is odd, $D_n$ has two irreducible characters of degree $1$ and $(n-1)/2$ of degree $2$, and when $n$ is even, there are four irreducible characters of degree $1$ and $(n-2)/2$ of degree $2$.  Hence if $(R,B)$ is the character ring of $D_n$ for $n\geq10$, $|G_R|<|B\setminus G_R|$, antithesis to when the nontrivial Frobenius-Perron dimension is irrational.
\item[(b)] Let $p\in\mathbb{Z}_{\geq2}$ be prime and $E$ an extraspecial $p$-group of order $p^{2n+1}$ for some $n\in\mathbb{Z}_{\geq1}$, i.e.\ $|Z(E)|=p$ and $E/Z(E)$ is an elementary abelian $p$-group of order $p^{2n}$.  For each $n\in\mathbb{Z}_{\geq1}$ there are exactly 2 isomorphism classes of extraspecial $p$-groups of order $p^{2n+1}$, and none whose order has an even exponent.  The irreducible characters of $E$ have two forms: $p^{2n}$ of degree $1$, and $p-1$ of degree $p^n$.  The characters of degree $p^n$ are fixed by the action of the linear characters by left (or right) multiplication, giving $p-1$ orbits under this action.  The \emph{semi-extraspecial} groups \cite{MR473004} provide more examples of finite groups whose irreducible characters have exactly two distinct degrees.  One can find an abstract study of finite groups with two distinct character degrees in \cite[Chapter 12]{MR2270898}.  As far as we can tell, at this time there is no complete alternative characterization of such finite groups. 
\end{itemize}
\end{example}

\par The following lemma is a collection of basic properties of two-dimension fusion rings $(R,B)$ while assuming $|B/ G_R|=2$ where $G_R$ is acting on $B$ by left multiplication.  For example, this is true when $\mathrm{FPdim}(x)\not\in\mathbb{Z}$ by Lemma \ref{transitive}.

\begin{lemma}\label{bigonea}
Let $(R,B)$ be a fusion ring with $G:=G_R$ and $|B/G_R|=2$.  Then for all $x,y\in B\setminus G$,
\begin{enumerate}
\item[(a)] if $g\in G$ such that $gx=x^\ast$, then $gx=xg$;
\item[(b)] $x^\ast\in x\cdot G$;
\item[(c)] $G\cdot x=x\cdot G$;
\item[(d)] if $y\in G\cdot x$, $xx^\ast=yy^\ast$ and $G_x=G_y$; 
\item[(e)] $G_x\trianglelefteq G$.
\end{enumerate}
\end{lemma}

\begin{proof}
If $gx=x^\ast$, then $x=g^{-1}x^\ast=(xg)^\ast$, hence $x^\ast=xg$, proving Parts (a) and (b).  If $y\in G\cdot x$, then there exists $g\in G$ such that $gx=y$, hence $x^\ast g^{-1}=y^\ast$.  By Part (b) there exist $h_x,h_y\in G$ such that $x^\ast=xh_x$ and $y^\ast=yh_y$, hence $x(h_xg^{-1}h_y^{-1})=y$, proving $G\cdot x\subseteq x\cdot G$; an identical argument proves the reverse inclusion, which proves Part (c).  If $y\in G\cdot x$, then by Part (c), $gx^\ast=y^\ast$ for some $g\in G$.  Hence
\begin{equation}\label{five}
yy^\ast=(y^\ast)^\ast y^\ast=(gx^\ast)^\ast gx^\ast=xg^{-1} gx^\ast=xx^\ast.
\end{equation}
The cyclic invariance of the fusion rules \cite[Proposition 3.1.6]{tcat} implies that for all $x\in B\setminus G$ and $g\in G$, $g^{-1}\in G_x$ if and only if $1=c^x_{g^\ast,x}=c_{x,x^\ast}^g$.  As $G_x$ is a subgroup this occurs if and only if $g\in G_x$.  Now using the above equality $xg^{-1}=y$, if $h\in G_x$, then $hy=hxg^{-1}=xg^{-1}=y$.  Thus $h\in G_y$ and since $x,y$ were arbitrary, $G_x=G_y$, completing the proof of Part (d).  Lastly, Part (d) implies for all $x\in B\setminus G$ and $g\in G$, $xx^\ast=gx(gx)^\ast=gxx^\ast g^{-1}$, thus by comparing the invertible summands, $\sum_{h\in G_x}h=\sum_{h\in G_x}ghg^{-1}$.  Thus $G_x$ is normal in $G$ and Part (e) is proven.
\end{proof}

%%%%%%%%%%%%%%%%%%%%%%%%%%%

%%%%%%%%%%%%%%%%%%%%%%%%%%%

\subsection{Representations of two-orbit fusion rings}\label{twoorbit}

If $(R,B)$ is a fusion ring and there exists $d\in\mathbb{R}\setminus\mathbb{Z}$ such that for all $x\in B$, $\mathrm{FPdim}(x)\in\{1,d\}$, the contrapositive of Lemma \ref{transitive} implies $|B/G_R|=2$.  This two-orbit property holds in many cases when $d\in\mathbb{Z}$ as well, and in either case, these \emph{two-orbit} fusion rings have abundant structure.  Since $G_x=G_y$ for all $x,y\in B\setminus G$ by Lemma \ref{bigonea}(d), we may label this distinguished normal subgroup by $H_R$.

\par For any finite group $G$ and subgroup $H\subset G$, let $\mathrm{Irr}_H(G)\subset\mathrm{Irr}(G)$ be the set of isomorphism classes of irreducible representations $\psi$ of $G$ such that $\psi|_H$ is a sum of nontrivial irreducible representations of $H$.  Let $1_K$ be the trivial representation of any subgroup $K\subset G$.  Frobenius reciprocity states that for any $\psi\in\mathrm{Irr}(G)$,
\begin{equation}
\mathrm{Hom}_{\mathrm{Rep}(H)}\left(1_H,\mathrm{Res}_H^G\psi\right)=\mathrm{Hom}_{\mathrm{Rep}(G)}\left(\mathrm{Ind}_H^G1_H,\psi\right)
\end{equation}
where $\mathrm{Res}_H^G$ is the restriction functor $\mathrm{Rep}(G)\to\mathrm{Rep}(H)$ and $\mathrm{Ind}_H^G$ its adjoint.  The representation $\mathrm{Ind}_H^G1_H$ is the permutation representation of $G$ acting on the cosets $G/H$.  It is productive to consider $\mathrm{Irr}_H(G)$ as the set of irreducible representations of $G$ which are \emph{not} summands of $\mathrm{Ind}_H^G1_H$.  The two extreme cases are: $\mathrm{Ind}_H^G1_H$ is the regular representation of $G$ when $H$ is trivial, hence $\mathrm{Irr}_H(G)$ is empty, and $\mathrm{Ind}_H^G1_H=1_G$ when $H=G$, hence $\mathrm{Irr}_H(G)=\mathrm{Irr}(G)\setminus\{1_G\}$.  Recall that $\dim(\mathrm{Ind}_H^G1_H)=[G:H]$.
\par Denote by $\mathrm{Irr}_0(R)\subset\mathrm{Irr}(R)$ the subset of the isomorphism classes of irreducible representations $\psi$ of $R$ such that $\psi(x)=0$ for all $x\in B\setminus G_R$.

\begin{proposition}\label{triv}
Let $(R,B)$ be a fusion ring with $|B/G_R|=2$.  There is a bijection between $\mathrm{Irr}_{H_R}(G_R)$ and $\mathrm{Irr}_0(R)$.  If $\psi\in\mathrm{Irr}_0(R)$, then $f_\psi=|G_R|/\dim(\psi)$.
\end{proposition}

\begin{proof}
Set $G:=G_R$, $H:=H_R$, and let $\psi:G\to\mathrm{End}(V)$ be a finite-dimensional irreducible representation.  Define $\tilde{\psi}:R\to\mathrm{End}(V)$ extended from $\psi$ such that $\tilde{\psi}(x):=0$ for all $x\in B\setminus G$.  If $x\in B\setminus G$, then $\tilde{\psi}(x)\tilde{\psi}(x^\ast)=0$.  And so $\tilde{\psi}\in\mathrm{Irr}_0(R)$ only if
\begin{equation}\label{eleven}
0=\tilde{\psi}(xx^\ast)=\sum_{g\in H}\psi(g).
\end{equation}
Among $\psi\in\mathrm{Irr}(G)$, condition (\ref{eleven}) is true if and only if $\psi\in\mathrm{Irr}_{H}(G)$.  Therefore for each $\psi\in\mathrm{Irr}_H(G)$, there exists a unique $\tilde{\psi}\in\mathrm{Irr}_0(R)$ such that $\tilde{\psi}|_G=\psi$.  Conversely, if $\psi\in\mathrm{Irr}_0(R)$, then any nontrivial decomposition of $\psi|_G$ into a sum of irreducible representations of $G$ is likewise a nontrivial decomposition of $\psi$ into a sum of irreducible representations of $R$.  Therefore $\psi|_G\in\mathrm{Irr}_H(G)$ for all $\psi\in\mathrm{Irr}_0(R)$ as (\ref{eleven}) must be satisfied.  Noting that $\psi|_{B\setminus G}=0$, one computes $f_\psi=|G|/\dim(\psi)$ from \cite[Lemma 2.3]{codegrees}.
\end{proof}

\begin{corollary}\label{corcom}
Let $(R,B)$ be a fusion ring with $|B/G_R|=2$ and $[G_R:H_R]=2$.  Then $R$ is commutative if and only if $G_R$ is abelian.
\end{corollary}

\begin{proof}
For the nontrivial implication, assume $G_R$ is abelian.  Then all $\psi\in\mathrm{Irr}(G_R)$ restrict to irreducible representations of $H_R$ and therefore $|\mathrm{Irr}_{H_R}(G_R)|=|G_R|-[G_R:H_R]$.  Therefore Proposition \ref{triv} states
\begin{equation}
|\mathrm{Irr}_0(R)|=|\mathrm{Irr}_{H_R}(G_R)|=|G_R|-[G_R:H_R]=|G_R|-2.
\end{equation}
The $1$-dimensional representations $\psi_\pm$ with $\psi_\pm|_{G_R}=1$ and $\psi_\pm(x)=d_\pm$ (the roots of Equation (\ref{eq:1})) do not lie in $\mathrm{Irr}_0(R)$, and we compute
\begin{equation}
\dim(\psi_+)^2+\dim(\psi_-)^2+\sum_{\psi\in\mathrm{Irr}_0(R)}\dim(\psi)^2=|G_R|.
\end{equation}
The rank of $R$ over $\mathbb{Z}$ is $|G_R|+2$, hence $\dim(\psi)^2\leq2$ for any other irreducible representation of $R$ not included in the above list, i.e.\ $\dim(\psi)=1$, thus $R$ is commutative.
\end{proof}

\begin{note}
Corollary \ref{corcom} is also proven directly in \cite[Proposition 3(1)]{vercleyen2022low}, but under slightly different assumptions.
\end{note}

\begin{note}\label{haagerup}
Corollary \ref{corcom} is false in generality when $[G_R:H_R]>2$.  To see this, let $G$ be any finite abelian group and define a two-orbit fusion ring $R$ with $G=G_R$ which acts fixed-point free on the remaining $|G|$ basis elements, i.e.\ $|B/G|=2$ and $H:=H_R$ is trivial.  Let $x$ be any one of the noninvertible basis elements.  The fusion rules are determined by the choice
\begin{equation}
(gx)(hx):=gh^{-1}+\sum_{g\in G}gx
\end{equation}
for all $g,h\in G$.  In particular, $g^{-1}x=xg$ and so such a fusion ring is commutative if and only if $|G|=[G:H]\leq2$.  Collectively these have been referred to as the \emph{Haagerup-Izumi} fusion rings in a categorical context (e.g.\ \cite{MR2837122}).  The smallest noncommutative example, with $G=C_3$, is known as the \emph{Haagerup fusion ring} and has had an important role in the history of fusion categories and subfactors (refer also to Example \ref{lastexamples}(b)).
\end{note}

The general representation theory of two-orbit fusion rings varies greatly depending on the fusion coefficients.  Mild assumptions about the groups $G_R$ and $H_R$, or about the fusion coefficients themselves, can produce a tractable set of examples.  One of the vital structures is a set of involutions relating $G_R$ and $H_R$ which exists thanks to the duality structure of fusion rings.  For $x\in B\setminus G_R$, define $\theta_x:G_R/H_R\to G_R/H_R$ by $\overline{g}\mapsto \overline{g'}$ where $g'x=xg$ and $\overline{g}$ is the coset $gH$ of $g\in G$ for brevity.

\begin{lemma}\label{theta}
Let $(R,B)$ be a fusion ring with $|B/G_R|=2$.  For all $x\in B\setminus G_R$, $\theta_x\in\mathrm{Aut}(G_R/H_R)$.
\end{lemma}

\begin{proof}
To prove $\theta_x$ is well-defined, let $g_1,g_2\in G$ such that $g_jx=xg$ for $j=1,2$.  Then $x^\ast g_2^{-1}=g^{-1}x^\ast$.  Hence  $g_1xx^\ast g_2^{-1}=g^{-1}x^\ast xg=g^{-1}xx^\ast g$.  By comparing invertible elements, there exists $h\in H_R$ such that $g_1hg_2^{-1}=e$, which is to say $\overline{g_1}=\overline{g_2}$. And if $\overline{g},\overline{h}\in G_R/H_R$, then $(gh)'x=x(gh)=g'xh=g'h'x$.  By multiplying on the right by $x^\ast$ and comparing invertible summands, this shows $(gh)'=fg'h'$ for some $f\in H_R$, proving $\theta$ is a group homomorphism.   The kernel of $\theta$ is trivial as $\overline{g}\in\mathrm{ker}\,\theta$ implies there exists $h\in H_R$ such that $x=hx=xg$, i.e.\ $x=xg$ which means $g\in H_R$ by definition.
\end{proof}

\begin{lemma}\label{sub}
Let $(R,B)$ be a fusion ring with $|B/G_R|=2$.  If $G_R/H_R$ is abelian, $\theta_x=\theta_y$ for all $x,y\in B\setminus G_R$, and $\theta_x^2=\mathrm{id}_{G_R/H_R}$.
\end{lemma}

\begin{proof}
Let $g\in G_R$.  If $hx=y$ for some $h\in G_R$, then $g'x=xg$ implies $hg'x=yg$.  But $hg'=fg'h$ for some $f\in H_R$ since $G_R/H_R$ is abelian.  Thus $(fg')y=yg$ with $\overline{fg'}=\overline{g'}$, i.e.\ $\theta_x(\overline{g})=\theta_y(\overline{g})$ for all $g\in G_R$.  Now $g'x=xg$ implies $xg^{-1}=(g')^{-1}x$, hence $gx^\ast=x^\ast g'$.  Therefore $\theta_{x^\ast}(\overline{g'})=\overline{g}$.  But we have shown $\theta_{x^\ast}=\theta_x$, proving $\theta_x$ has order at most 2.
\end{proof}

%%%%%%%%%%%%%%%%%%%%%%%

%%%%%%%%%%%%%%%%%%%%%%%

%%%%%%%%%%%%%%%%%%%%%%%

\subsection{Representations of uniform two-orbit fusion rings}\label{sec:uni}

\par We end our section on fusion rings with a slight deviation from our main goal in order to discuss one class of two-orbit fusion rings which has particularly elegant representation theory.  These are unsurprisingly the vast majority of two-orbit fusion rings which are known to be categorifiable.  Let $(R,B)$ be a fusion ring with $|B/G_R|=2$, and denote $G:=G_R$ and   $H:=H_R$.  We say a fusion ring $(R,B)$ with $|B/G_R|=2$ is \emph{uniform}, or $R$ has uniform fusion rules, if $G/H$ is abelian, and there exists $k\in\mathbb{Z}_{\geq0}$ such that for all $x\in B\setminus G$,
\begin{equation}\label{thurteen}
xx^\ast=\sum_{g\in H}g+k\sum_{y\in B\setminus G}y.
\end{equation}
Of most importance in the definition of uniform is that there is a canonical automorphism $\theta_R\in\mathrm{Aut}(G/H)$ which is given by $\theta_x$ for any $x\in B\setminus G$ by Lemmas \ref{theta}--\ref{sub}.  We believe this condition plays an important role in the categorification of two-orbit fusion rings.

\begin{example}\label{fooor}\,
\begin{itemize}
\item[(a)]  In \cite{MR4179724}, Pinhas Grossman \& Masaki Izumi produced infinite families of modular data \cite[Section 8.17]{tcat} which are conjecturally associated with \emph{quadratic} fusion categories, which are in turn categorifications of two-orbit fusion rings.  The Grothendieck rings of quadratic fusion categories in this sense are a subset of the uniform fusion rings described here where $\theta_R$ is assumed to be the inversion automorphism, the uniform fusion coefficient $k$ is a multiple of $|H|$, and $G$ itself is abelian.  But for instance the character ring of $S_3$ is categorifiable and not quadratic, but is uniform.  An important example of a categorifiable two-orbit uniform fusion ring which is not quadratic with $G_R$ nonabelian is provided in \cite[Example 5.1(3)]{vercleyen2022low}.  In this case $G_R=S_3$, $H_R=C_3$, and $k=1$.
\item[(b)]  Let $(R,B)$ be a uniform two-orbit fusion ring.  Identify $K:=G_R/H_R$ with a subset of coset representatives in $G_R$, i.e.\ for any $x\in B\setminus G_R$, $B\setminus G_R=\{gx:g\in K\}$.  In this manner $\theta_R$ can be considered as a permutation of the set $K$.  If $k=0$, then $(gx)(hx)=g\theta_R(h)\sum_{f\in H_R}f$ for all $g,h\in K$, e.g.\ when $H_R$ is trivial, $R$ is isomorphic to the integral group ring $\mathbb{Z}(G_R\rtimes_{\theta_R}C_2)$.  Proposition \ref{reps} below states that the representation-theoretic properties of these rings remains essentially constant when the elements of $B\setminus G_R$ are self-dual and $G_R$ is abelian, independent of the parameter $k\in\mathbb{Z}_{\geq0}$.  So in at least this case, one may think of the two-orbit uniform fusion rings as simply ``deformations'' of the finite groups $(G_R/H_R)\rtimes_{\theta_R}C_2$ of varying intensity.
\end{itemize}
\end{example}

\begin{proposition}\label{reps}
Let $(R,B)$ be a fusion ring with $|B/G_R|=2$.  Assume $G:=G_R$ is abelian and denote $H:=H_R$.   If $R$ is uniform and the elements of $B\setminus G$ are self-dual, then there is a bijection between $\mathrm{Irr}(R)$ and the disjoint union $\mathrm{Irr}_H(G)\cup\mathrm{Irr}((G/H)\rtimes_{\theta_R} C_2)$.
\end{proposition}

\begin{proof}
Identify $K:=G/H$ with a subset of coset representatives in $G$ so that $\theta_R$ can be considered as a permutation of the set $K$, and let $k\in\mathbb{Z}_{\geq0}$ be the uniform multiplicity of the fusion rules of $R$ as in Equation (\ref{thurteen}).  Lemma \ref{triv} states that the $\psi\in\mathrm{Irr}(R)$ which vanish on $B\setminus G$ are in bijection with $\mathrm{Irr}_H(G)$.  Since $G$ is assumed abelian, then $|\mathrm{Irr}_H(G)|=|G|-[G:H]$.  Now let $\psi:(G/H)\rtimes_{\theta_R} C_2\to\mathrm{End}(V)$ be irreducible such that $\psi|_{G/H}$ is a sum of nontrivial irreducible representations of $G/H$.  By Frobenius-reciprocity, $1_{G/H}$ appears as a summand of $\psi|_{G/H}$ if and only if $\psi$ is a summand of $\mathrm{Ind}_{G/H}^{(G/H)\rtimes_{\theta_R}C_2}1_{G/H}$, which is $2$-dimensional.  Therefore only two $1$-dimensional representations have this property: the representations $1_{(G/H)\rtimes_{\theta_R}C_2}^\pm$ where $1_{(G/H)\rtimes_{\theta_R}C_2}^+:=1_{(G/H)\rtimes_{\theta_R}C_2}$ and $1_{(G/H)\rtimes_{\theta_R}C_2}^-$ is its complement in $\mathrm{Ind}_{G/H}^{(G/H)\rtimes_{\theta_R}C_2}1_{G/H}$.   So our assumption is that $\psi$ is distinct from these.

\par Denote $C_2=\{0,1\}$, and define $\tilde{\psi}:R\to\mathrm{End}(V)$ such that $\tilde{\psi}(g):=\psi(\overline{g}\times 0)$ for all $g\in G$ (where $\overline{g}$ is the coset of $g$ in $G/H$), and for a fixed $x\in B\setminus G$, $\tilde{\psi}(gx):=\sqrt{|H|}\psi(\overline{g}\times 1)$ for $g\in K$, extending linearly.  As $x=x^\ast$, this implies
\begin{equation}
\tilde{\psi}(x^2)=\tilde{\psi}\left(\sum_{g\in H}g\right)+k\tilde{\psi}\left(\sum_{g\in K}gx\right)=|H|\mathrm{id}_V+k\sum_{g\in G/H}\psi(\overline{g}\times0)=|H|\mathrm{id}_V
\end{equation}
since $\psi\in\mathrm{Irr}_{G/H}((G/H)\rtimes_{\theta_R} C_2)$ by assumption.  To finish verifying that $\tilde{\psi}$ is a representation of $R$, note that $\tilde{\psi}(gh)=\tilde{\psi}(g)\tilde{\psi}(h)$ for any $g,h\in G$ and $\tilde{\psi}(gx)=\tilde{\psi}(g)\tilde{\psi}(x)$ trivially from the definition. It remains to verify that for all $g,h\in K$
\begin{align}
\tilde{\psi}(gxhx)=\tilde{\psi}(g\theta_R(h)x^2)=\tilde{\psi}(g\theta_R(h))\tilde{\psi}(x^2)&=|H|\tilde{\psi}(\overline{g\theta_R(h)}) \\
&=|H|\psi(\overline{g\theta_R(h)}\times 0) \\
&=|H|\psi((\overline{g}\times 1)(\overline{h}\times1)) \\
&=\sqrt{|H|}\psi(\overline{g}\times 1)\sqrt{|H|}\psi(\overline{h}\times1)) \\
&=\tilde{\psi}(gx)\tilde{\psi}(hx).
\end{align}
In particular, each $\psi\in\mathrm{Irr}((G/H)\rtimes_{\theta_R}C_2)$ determines at least one irreducible representation of $R$, if one associates the trivial and sign representations of $(G/H)\rtimes_{\theta_R}C_2$ to $R$ with the two unique 1-dimensional $\psi\in\mathrm{Irr}(R)$ such that $\psi|_G=1$ and $\psi(x)=\pm d$ for $x\in B\setminus G$.  The latter two representations of $R$ are clearly nonisomorphic, while the remaining representations are pairwise nonisomorphic because their restrictions to $G/H$, hence to $G$, are.  Lastly, we compute
\begin{equation}
\sum_{\psi\in\mathrm{Irr}_H(G)\cup\mathrm{Irr}((G/H)\rtimes_{\theta_R}C_2)}\dim(\psi)^2=|G|-[G:H]+2[G:H]=\mathrm{rank}(R),
\end{equation}
and so the above list of irreducible representations of $R$ is exhaustive.
\end{proof}

\begin{example}\label{lastexamples}\,
\begin{itemize}
\item[(a)] The \emph{near-group} fusion rings $R(G,\ell)$ are those uniform two-orbit fusion rings with $G_R=H_R$.  Their representation theory was described in \cite[Appendix A]{ost15} which one verifies against Proposition \ref{reps} when $G_R$ is abelian since $(G_R/H_R)\rtimes_{\theta_R} C_2\cong C_2$, and $\mathrm{Irr}_{H_R}(G_R)=\mathrm{Irr}(G_R)\setminus\{1_G\}$.
\item[(b)] The Haagerup fusion ring (see Note \ref{haagerup}) has $G_R=C_3$, $H_R$ trivial, $\theta_R$ the inversion automorphism of $C_3$, and uniform fusion rules with $k=1$.  A connection between a categorification of the Haagerup fusion ring to the symmetric group $S_3$ was first noted by David Evans in \cite{MR2343850}, and elaborated upon by David Evans \& Terry Gannon in \cite{MR2837122} for the Haagerup-Izumi series of fusion rings with $G_R=C_n$.  In particular, the modular data of the doubles of these fusion categories (refer to Section \ref{sec:cat}) can be understood in part in terms of the modular data of the twisted doubles of $C_n\rtimes_{\theta_R} C_2\cong D_n$.  Proposition \ref{reps} demonstrates that this connection exists on the level of the representations of their Grothendieck rings, and furthermore, independent of the uniform fusion coefficient $k\in\mathbb{Z}_{\geq0}$.
\end{itemize}
\end{example}

%%%%%%%%%%%%%%%%%%%%%%%%%

%%%%%%%%%%%%%%%%%%%%%%%%%

%%%%%%%%%%%%%%%%%%%%%%%%%

\section{Categorification}\label{sec:cat}

\par A fusion category is the categorical analog of a fusion ring.  As we are mainly interested in which fusion rings arise as the Grothendieck rings of fusion categories, we will only briefly recount the basic notions of fusion categories and direct the interested reader to the standard text \cite{tcat} for additional detail.  Succinctly, a fusion category (over $\mathbb{C}$) is a $\mathbb{C}$-linear semisimple rigid monoidal category (with product $\otimes$, monoidal unit $\mathbbm{1}_\mathcal{C}$, and duality $X\mapsto X^\ast$) whose set of isomorphism classes of simple objects, $\mathcal{O}(\mathcal{C})$, is finite and includes $\mathbbm{1}_\mathcal{C}$.  This guarantees the Grothendieck ring of $\mathcal{C}$ is a fusion ring, and so all of the concepts/notation/vocabulary related to fusion rings from Section \ref{sec:fus} effortlessly pass to fusion categories.  A pivotal structure \cite[Section 4.7]{tcat} for a fusion category allows one to produce a robust set of numerical invariants through a categorical notion of trace \cite[Section 4.7]{tcat}, denoted $\mathrm{Tr}(-)$ in our case, including categorical dimensions $\dim(X)$ and Frobenius-Schur indicators $\nu_n(X)$ \cite{MR2313527} for all $X\in\mathcal{O}(\mathcal{C})$ and $n\in\mathbb{Z}_{\geq1}$.  The details of pivotal structures are less relevant in this exposition as we immediately observe (Lemma \ref{immediate}) that the fusion categories in question here possess a canonical pivotal structure, which is spherical in the sense of \cite[Definition 4.7.14]{tcat}.
\par In particular, the double (or Drinfeld center \cite[Section 7.13]{tcat}) $\mathcal{Z}(\mathcal{C})$ of a spherical fusion category is a modular tensor category \cite[Section 8.14]{tcat}, which is a fusion category equipped with additional structure such as a braiding \cite[Section 8.1]{tcat}.  One other structure that is necessary for our arguments is the collection of \emph{twists} $\theta_X$ which are roots of unity for all simple $X$ in a modular tensor category \cite[Section 8.18]{tcat}.  With all of these novel tools, it is often easier to understand fusion categories by inspecting their doubles rather than analyzing the fusion category directly.  On the level of objects, the double $\mathcal{Z}(\mathcal{C})$ of a fusion category $\mathcal{C}$ consists of pairs of an object $X\in\mathcal{C}$ and commutation data for all other objects of $\mathcal{C}$ (so-called half-braidings).  A single object of $\mathcal{C}$ may commute with other objects in myriad ways, so the double of a fusion category is larger; specifically $\mathrm{FPdim}(\mathcal{Z}(\mathcal{C}))=\mathrm{FPdim}(\mathcal{C})^2$ \cite[Theorem 7.16.6]{tcat}.  By forgetting the half-braidings, there is a functor $F:\mathcal{Z}(\mathcal{C})\to\mathcal{C}$, whose adjoint we refer to as \emph{induction} $I:\mathcal{C}\to\mathcal{Z}(\mathcal{C})$ \cite[Section 9.2]{tcat}.  The adjointness condition implies that for all $X\in\mathcal{C}$ and $Y\in\mathcal{Z}(\mathcal{C})$, $\mathrm{Hom}_{\mathcal{Z}(\mathcal{C})}(I(X),Y)=\mathrm{Hom}_{\mathcal{C}}(X,F(Y))$.  For objects $X,Y$ in a fusion category $\mathcal{C}$, we will repeatedly use the notation
\begin{equation}
[X,Y]:=\dim\mathrm{Hom}_\mathcal{C}(X,Y)
\end{equation}
for brevity, where the category itself is suppressed in the notation.  This will not be an issue as we are only considering fusion categories and their doubles, and the ambient category will be clear by context.
\par A brief review of the needed results about the induction and restriction functors, and their relation to formal codegrees, and the twists in the double can be found in \cite[\mbox{Sections 2.2--2.3}]{ost15}.  Of most importance are the facts that \cite[Proposition 5.4]{ENO}: if $\mathcal{C}$ is a fusion category, for all $X\in\mathcal{O}(\mathcal{C})$
\begin{equation}
F(I(X))=\bigoplus_{Y\in\mathcal{O}(\mathcal{C})}Y\otimes X\otimes Y^\ast,
\end{equation}
and \cite[Theorem 2.13]{ost15} if $f_\psi$ is a formal codegree of a fusion ring categorified by a spherical fusion category $\mathcal{C}$, then there exists $A_\psi\in\mathcal{O}(\mathcal{Z}(\mathcal{C}))$ such that $f_\psi\dim(A_\psi)=\dim(\mathcal{C})$ and $[\mathbbm{1},F(A_\psi)]=[I(\mathbbm{1}_\mathcal{C}),A_\psi]=\dim(\psi)$.

\subsection{Results for two-dimension and two-orbit fusion rings}\label{sec:catstuff}

Recall that a fusion category $\mathcal{C}$ is \emph{pseudounitary} if $\mathrm{FPdim}(\mathcal{C})=\dim(\mathcal{C})$, in the sense of \cite[Definition 2.2]{ENO}.  This ensures $\mathcal{C}$ possesses a unique spherical structure such that $\dim(X)=\mathrm{FPdim}(X)$ for all $X\in\mathcal{C}$ \cite[Proposition 9.5.1]{tcat}.  The following lemma states that this is the case for all categorifications of two-dimension fusion rings and we will henceforth assume $\mathcal{C}$ has been equipped with this canonical spherical structure, but we will still favor $\mathrm{FPdim}$ in our notation to distinguish from the (complex) dimensions of representations.

\begin{lemma}\label{immediate}
Let $(R,B)$ be a fusion ring and $d\in\mathbb{R}_{>1}$ such that for all $x\in B$, $\mathrm{FPdim}(x)\in\{1,d\}$.  If $R$ is categorified by a fusion category $\mathcal{C}$, then $\mathcal{C}$ is pseudounitary.
\end{lemma}

\begin{proof}
Lemma \ref{transitive} states that if $d$ is irrational, then $|B/G_R|=2$, thus $\mathcal{C}$ is pseudounitary by \cite[Theorem IV.3.6]{thornton2012generalized}.  Otherwise $\mathcal{C}$ is integral, hence pseudounitary by \cite[Proposition 9.6.5]{tcat}.
\end{proof}

\begin{proposition}\label{noncom}
Let $(R,B)$ be a fusion ring with $|B/G_R|=2$. Assume that the nontrivial Frobenius-Perron dimension $d$ is a root of $x^2-rx-s$ for some $r\in\mathbb{Z}_{\geq0}$ and $s\in\mathbb{Z}_{\geq1}$, and that $0<r<|H_R|-1$.  If $R$ is categorifiable, then $R$ is noncommutative.
\end{proposition}

\begin{proof}
We will prove the the statement by contradiction.  Label $G:=G_R$ and $H:=H_R$.  Assume $R$ is commutative and categorifiable.  If $\psi\in\mathrm{Irr}(R)$, then $f_\psi\geq|G|$ by Lemma \ref{bound}, hence in the notation of \cite[Theorem 2.13]{ost15}, any categorification $\mathcal{C}$ of $R$ satisfies
\begin{align}
\mathrm{FPdim}(A_\psi)=\dfrac{\mathrm{FPdim}(R)}{f_\psi}\leq\dfrac{|G|+[G:H]d^2}{|G|}=1+\dfrac{d^2}{|H|}&=1+1+\dfrac{r}{|H|}d \\
&<2+d-\dfrac{d}{|H|}
\end{align}
where the last inequality follows from $r<|H|-1$.  We have $1+d-d/|H|\leq d$, thus $\mathrm{FPdim}(A_\psi)\leq1+d$ with equality if and only if $d=|H|$ which occurs if and only if $r=|H|-1$.   Therefore $\mathrm{FPdim}(A_\psi)<1+d$, and thus any categorification $\mathcal{C}$ of $R$ has $[X,F(A_\psi)]=0$ for all $X\in\mathcal{O}(\mathcal{C})$ with $\mathrm{FPdim}(X)=d$ since $[\mathbbm{1},F(A_\psi)]=\dim(\psi)=1$.  But this contradicts the fact that $F(I(\mathbbm{1}))$ contains at least one noninvertible summand from the assumption that $r>0$.
 
\end{proof}

\begin{proposition}\label{prop1}
Let $(R,B)$ be a fusion ring and $d\in\mathbb{R}\setminus\mathbb{Z}$ be a root of $x^2-rx-s$ for some $r\in\mathbb{Z}_{\geq0}$ and $s\in\mathbb{Z}_{\geq1}$ such that for all $x\in B$, $\mathrm{FPdim}(x)\in\{1,d\}$ .  If $R$ is categorifiable, then $s$ divides $r$.
\end{proposition}

\begin{proof}
The irrationality of $d$ implies $|B/G_R|=2$ by Lemma \ref{lem1}, so let $G:=G_R$ and $H:=H_R$.  Let $\sigma\in\mathrm{Gal}(\mathbb{Q}(d)/\mathbb{Q})$ be the nontrivial automorphism, so that $d+\sigma(d)=r$ and $d\sigma(d)=-s$.  We compute
\begin{align}
\dfrac{\mathrm{FPdim}(R)}{[G:H]s}&=\dfrac{|G|+[G:H]d^2}{-[G:H]d\sigma(d)}=1-\dfrac{d}{\sigma(d)}
\end{align}
\begin{align}
\dfrac{\sigma(\mathrm{FPdim}(R))}{[G:H]s}&=\dfrac{|G|+[G:H]\sigma(d^2)}{-[G:H]d\sigma(d)}=1-\dfrac{\sigma(d)}{d}.
\end{align}
Thus
\begin{equation}\label{lab}
\dfrac{\mathrm{FPdim}(R)}{\sigma(\mathrm{FPdim}(R))}=\dfrac{1-\dfrac{d}{\sigma(d)}}{1-\dfrac{\sigma(d)}{d}}=-\dfrac{d}{\sigma(d)}=\dfrac{d^2}{s}=1+\dfrac{r}{s}d.
\end{equation}
Now $\sigma(\mathrm{FPdim}(R))$ is a formal codegree of $R$ corresponding to some $\psi\in\mathrm{Irr}(R)$.  We compute in the notation of \cite[Theorem 2.13]{ost15} that $\mathrm{FPdim}(A_\psi)=\mathrm{FPdim}(\mathcal{C})/\sigma(\mathrm{FPdim}(\mathcal{C}))=1+(r/s)d$.  But $F(A_\psi)$ must decompose into invertible objects and simple objects of dimension $d$.  As $d$ is irrational, this implies $r/s\in\mathbb{Z}$.
\end{proof}

\par When $(R,B)$ is a categorifiable fusion ring with $|B/G_R|=2$ and $d$ is irrational, Proposition \ref{prop1} allows many computations to simplify since $d$ is a root of $x^2-k|H_R|x-|H_R|$ for some $k\in\mathbb{Z}_{\geq0}$, and
\begin{equation}
\mathrm{FPdim}(R)=|G_R|+[G_R:H_R]d^2=|G_R|(2+kd).
\end{equation}

\begin{proposition}
Let $(R,B)$ be a fusion ring with $|B/G_R|=2$ such that $d$ is irrational.  If $R$ is categorifiable, then $\dim(\psi)\leq[G_R:H_R]$ for all $\psi\in\mathrm{Irr}_H(G)$.
\end{proposition}

\begin{proof}
Each $\psi\in\mathrm{Irr}_{H_R}(G_R)$ corresponds to a $\psi\in\mathrm{Irr}(R)$ with $f_\psi=|G_R|/\dim(\psi)$ by Proposition \ref{triv}.  We compute
\begin{equation}
\dim(A_\psi)=\dfrac{\mathrm{FPdim}(R)}{f_\psi}=2\dim(\psi)+k\dim(\psi)d.
\end{equation}
We know $[F(A_\psi),\mathbbm{1}]=\dim(\psi)$ \cite[Theorem 2.13]{ost15}, hence $F(A_\psi)$ contains at least one nontrivial invertible summand as $d$ is irrational.  For all $g\in G_R$,
\begin{equation}[F(I(\mathbbm{1})),g]\geq[F(A_\psi),\mathbbm{1}][F(A_\psi),g]=\dim(\psi)[F(A_\psi),g].
\end{equation}
But
\begin{equation}
F(I(\mathbbm{1}))=|G_R|+[G_R:H_R]\bigoplus_{h\in H_R}h+Z
\end{equation}
where $Z$ is a sum of noninvertible simple objects.  Hence $[F(I(\mathbbm{1})),g]=[G_R:H_R]$ if $g\in H_R$ and $[F(I(\mathbbm{1})),g]=0$ otherwise.  Moreover $\dim(\psi)\leq[G_R:H_R]$.
\end{proof}

\begin{note}
In the case of near-group fusion categories, we observe the well-known fact that $d$ irrational implies $G_R=H_R$ is abelian since $[G_R:H_R]=1$.
\end{note}

Another useful property of modular tensor categories is a Galois action on the set of isomorphism classes of simple objects \cite{gannoncoste}.  If $\mathcal{C}$ is a modular tensor category, the Galois action on simple objects arises from a canonical bijection between the simple objects and the characters of the Grothendieck ring, whose values lie in a cyclotomic field $\mathbb{Q}(\zeta_N)$ where $N\in\mathbb{Z}_{\geq1}$ is the conductor of $\mathcal{C}$ \cite[Theorem 8.14.7]{tcat}.  In particular, if $X\in\mathcal{O}(\mathcal{C})$ for a modular tensor category $\mathcal{C}$, and $\sigma\in\mathrm{Gal}(\mathbb{Q}(\zeta_N)/\mathbb{Q})$, then we denote the Galois permutation associated with $\sigma$ as $\hat{\sigma}:\mathcal{O}(\mathcal{C})\to\mathcal{O}(\mathcal{C})$.  By considering the Galois orbit of $\mathbbm{1}_\mathcal{C}$ (assuming $\mathcal{C}$ is pseudounitary), one has
\begin{equation}\label{twentyseven}
\mathrm{FPdim}(X)^2=\dfrac{\mathrm{FPdim}(\mathcal{C})}{\sigma(\mathrm{FPdim}(\mathcal{C}))}\sigma(\mathrm{FPdim}(X)^2).
\end{equation}

\begin{lemma}\label{gal}
Let $(R,B)$ be a fusion ring with $|B/G_R|=2$ such that $d$ is irrational.  Assume $R$ is categorified by a fusion category $\mathcal{C}$, so that $d$ is a root of $x^2-k|H|x-|H|$ for some $k\in\mathbb{Z}_{\geq0}$ by Proposition \ref{prop1}.  If $\mathrm{FPdim}(Y)=a+bd$ for some $Y\in\mathcal{O}(\mathcal{Z}(\mathcal{C}))$ and $a,b\in\mathbb{Z}_{\geq0}$, then $\mathrm{FPdim}(\hat{\sigma}(Y))=a+(ak-b)d$.  In particular $0\leq b\leq ak$.
\end{lemma}

\begin{proof}
We have $\sigma(d)=k|H|-d$ and thus using Equations (\ref{lab}) and (\ref{twentyseven}),
\begin{align}
\mathrm{FPdim}(\hat{\sigma}(Y))^2=\dfrac{\mathrm{FPdim}(\mathcal{C})}{\sigma(\mathrm{FPdim}(\mathcal{C}))}\sigma(\mathrm{FPdim}(Y)^2)&=(1+kd)(a+b(k|H|-d))^2 \\
&=(a+(ak-b)d)^2.
\end{align}
\end{proof}

%%%%%%%%%%%%%%%%%%%%%%%%%%%%%%
%%%%%%%%%%%%%%%%%%%%%%%%%%%%%%

\subsection{Twists induced from invertible objects}\label{sec:twist}

Categorifiable two-orbit fusion rings are amenable to study for many reasons, including their simplified representation theory as shown in Section \ref{twoorbit}. Their categorifications are also approachable since a relatively large proportion of their simple objects are invertible.  It is well-known, for example, that the twists of the simple summands of $I(\mathbbm{1}_\mathcal{C})$ are trivial for any spherical fusion category $\mathcal{C}$ (see \cite[Theorem 2.5]{ost15} or \cite[Remark 4.6]{MR2313527}).  This observation is a specific instance of the following facts, generalized to the induction of arbitrary invertible objects.  Make particular note that we do not assume $\mathcal{C}$ is a pseudounitary fusion category in this section, so the categorical notion of dimension is used \cite[Section 2.1]{ENO} in lieu of Frobenius-Perron dimension.

\begin{lemma}\label{wan}
Let $\mathcal{C}$ be a pivotal fusion category and $X\in\mathcal{O}(\mathcal{C})$.  If $X^{\otimes n}=\mathbbm{1}_\mathcal{C}$ for some $n\in\mathbb{Z}_{\geq1}$, then $\nu_n(X)$ is a root of unity of order $n$.
\end{lemma}

\begin{proof}
The Frobenius-Schur indicator $\nu_n(X)$ is defined \cite[Definition 3.1]{MR2381536} as the trace of an endomorphism $E_X^{(n)}$ of $\mathrm{Hom}_\mathcal{C}(\mathbbm{1}_\mathcal{C},X^{\otimes m})=\mathrm{Hom}_\mathcal{C}(\mathbbm{1}_\mathcal{C},\mathbbm{1}_\mathcal{C})=\mathbb{C}$.  In particular $E_X^{(n)}=\lambda\cdot\mathrm{id}_\mathbb{C}$ for some $\lambda\in\mathbb{C}$.   But \cite[Theorem 5.1]{MR2381536} states that
\begin{equation}
\mathrm{id}_\mathbb{C}=(E_X^{(n)})^n=\lambda^n\cdot\mathrm{id}_\mathbb{C}.
\end{equation}
Therefore $\lambda^n=1$ and moreover $\nu_n(X)^n=(\mathrm{Tr}(\lambda\cdot\mathrm{id}_\mathbb{C}))^n=\lambda^n=1$.
\end{proof}

\begin{proposition}\label{thelemma}
Let $\mathcal{C}$ be a spherical fusion category and $X\in\mathcal{O}(\mathcal{C})$.  If $X^{\otimes n}=\mathbbm{1}_\mathcal{C}$ for some $n\in\mathbb{Z}_{\geq1}$, then $\theta^n_Y=\theta_{Y'}^n$ for all simple summands $Y,Y'\subset I(X)$ in $\mathcal{Z}(\mathcal{C})$.
\end{proposition}

\begin{proof}
Using \cite[Theorem 5.1]{MR2313527}, $\dim(\mathcal{C})\nu_n(X)=\mathrm{Tr}(\theta_{I(X)}^n)$.  Hence Lemma \ref{wan} implies
\begin{align}
\dim(\mathcal{C})=|\dim(\mathcal{C})\nu_n(X)|&=\left|\sum_{Y\in\mathcal{O}(\mathcal{Z}(\mathcal{C}))}\dim(Y)[Y,I(X)]\theta^n_Y\right| \\
&\leq\sum_{Y\in\mathcal{O}(\mathcal{Z}(\mathcal{C}))}\dim(Y)[Y,I(X)] \\
&=\dim(I(X))=\dim(X)\dim(\mathcal{C})=\dim(\mathcal{C})
\end{align}
by the triangle inequality, where the last line follows from \cite[Proposition 5.4]{ENO}.  Equality exists if and only if the summands $\dim(Y)[Y,I(X)]\theta_Y^n$ over all $Y\in\mathcal{O}(\mathcal{Z}(\mathcal{C}))$ have the same argument, i.e.\ $\theta_Y^n=\theta_{Y'}^n$ for all simple $Y,Y'\subset I(X)$.
\end{proof}

\begin{note}
The fact that for a spherical fusion category $\mathcal{C}$, $\theta_Y=1$ for all simple $Y\subset I(\mathbbm{1}_\mathcal{C})$ is a trivial consequence since $\mathbbm{1}_{\mathcal{Z}(\mathcal{C})}\subset I(\mathbbm{1}_\mathcal{C})$ and $\theta_{\mathbbm{1}_{\mathcal{Z}(\mathcal{C})}}=1$.
\end{note}

\begin{corollary}\label{negcor}
Let $\mathcal{C}$ be a spherical fusion category and $X\in\mathcal{O}(\mathcal{C})$ such that $X^{\otimes2}\cong\mathbbm{1}_\mathcal{C}$.  Then $\theta_Y=\pm1$ for all $Y\subset I(X)$ if $\nu_2(X)=1$, or $\theta_Y=\pm i$ for all $Y\subset I(X)$ if $\nu_2(X)=-1$.
\end{corollary}

\begin{proof}
We have by \cite[Theorem 5.1]{MR2313527},
\begin{equation}
\nu_2(X)\dim(\mathcal{C})=\mathrm{Tr}(\theta_{I(X)}^2)=\sum_{Y\subset I(X)}\dim(Y)\theta_Y^2=\zeta\sum_{Y\subset I(X)}\dim(Y)
\end{equation}
for some root of unity $\zeta$ by Proposition \ref{thelemma}.  When $X\cong X^\ast$, we have $\nu_2(X)=\pm1$ by \cite[Theorem 2.7]{ost15}.  As $\dim(Y)\in\mathbb{R}$ for all $Y\subset I(X)$, $\zeta=\theta_Y^2=1$ for all $Y\subset I(X)$ when $\nu_2(X)=1$, and $\zeta=\theta_Y^2=-1$ for all $Y\subset I(X)$ when $\nu_2(X)=-1$.
\end{proof}

\begin{proposition}\label{theprop}
Let $\mathcal{C}$ be a spherical fusion category and $X\in\mathcal{O}(\mathcal{C})$ such that $X^{\otimes n}=\mathbbm{1}_\mathcal{C}$ for some $n\in\mathbb{Z}_{\geq1}$.  If the left or right $\otimes$-action of $X$ on $\mathcal{O}(\mathcal{C})$ has a fixed-point, then $\theta_Y$ is an $n$th root of unity for all simple $Y\subset I(X)$.
\end{proposition}

\begin{proof}
Recall \cite[Proposition 5.4]{ENO} that
\begin{equation}
F(I(X))=\bigoplus_{Z\in\mathcal{O}(\mathcal{C})}Z\otimes X\otimes Z^\ast.
\end{equation}
If $Z\otimes X=Z$ or $X\otimes Z^\ast=Z^\ast$ for some $Z\in\mathcal{O}(\mathcal{C})$, then $\mathbbm{1}_\mathcal{C}\subset Z\otimes Z^\ast\subset F(I(X))$.  Therefore $0<[\mathbbm{1}_\mathcal{C},F(I(X))]=[I(\mathbbm{1}_\mathcal{C}),I(X)]$.  Moreover there exists simple $Y\subset I(X)\cap I(\mathbbm{1}_\mathcal{C})$ such that $\theta_Y=1$.  The result then follows from Proposition \ref{thelemma}.
\end{proof}

\begin{note}
Proposition \ref{theprop} fails for some fixed-point free invertible objects.  Consider either pointed modular tensor category $\mathcal{C}$ of rank 2, with nontrivial simple object $X$.  The $\otimes$-action of $X$ on $\mathcal{O}(\mathcal{C})$ is fixed-point free and $I(X)=Y_1+Y_2$ where $Y_1$ and $Y_2$ are invertible objects with twists $i$ and $-i$.  These are fourth roots of unity while $X^{\otimes2}=\mathbbm{1}_\mathcal{C}$.
\end{note}

\begin{corollary}
Let $\mathcal{C}$ be a spherical fusion category and $X\in\mathcal{O}(\mathcal{C})$.  If $X^{\otimes2}\cong\mathbbm{1}_\mathcal{C}$ and $\nu_2(X)=-1$, then the action $X\otimes -$ (or $-\otimes X$) on $\mathcal{O}(\mathcal{C})$ is fixed-point free.
\end{corollary}

\begin{proof}
This follows from Corollary \ref{negcor} and the contrapositive of Proposition \ref{theprop}.
\end{proof}

%%%%%%%%%%%%%%%%%%%%%%%%%%%%%%%
%%%%%%%%%%%%%%%%%%%%%%%%%%%%%%%

\section{Near-group fusion categories}\label{sec:near}

Our terminal section is an application of most all of the previous results to produce bounds for the levels $\ell\in\mathbb{Z}_{\geq0}$ such that the near-group fusion rings $R(G,\ell)$ are categorifiable.

\subsection{Labelling simple objects of the double with $G$ abelian}\label{sec:label}

\par Let $\mathcal{C}$ be a near-group fusion category, i.e.\ a categorification of $R(G,\ell)$ for a finite group $G$ and level $\ell\in\mathbb{Z}_{\geq0}$, which we may assume is pseudounitary without loss of generality by Lemma \ref{immediate}.   We also assume $G$ is abelian, as nonabelian $G$ already have a level bound ($\ell<|G|$) by \cite[Theorem A.6]{ost15}, and a complete classification of categorifiable levels if one demands the categorifications be unitary \cite[Theorem 6.1]{MR3635673}.
\par There are three classes of simple objects of $\mathcal{Z}(\mathcal{C})$ which can be discovered via the induction functor $I:\mathcal{C}\to\mathcal{Z}(\mathcal{C})$, adjoint to the forgetful functor $F:\mathcal{Z}(\mathcal{C})\to\mathcal{C}$ \cite[Section 5.8]{ENO}.  One computes via \cite[Theorem 5.4]{ENO} for all $g\in G$,
\begin{equation}\label{invertible}
F(I(g))=|G|g\oplus \ell\rho\oplus\bigoplus_{g\in G}g
\end{equation}
where $\rho$ is (the unique, up to isomorphism) noninvertible.  As a consequence of the adjointess of $F$ and $I$, $0<[Y,I(g)]=[F(Y),g]$.  Now let $Y\subset I(g)$ be simple such that $h\subset F(Y)$ for $h\neq g$.  At least one such simple summand exists for all $h\neq g$ due to (\ref{invertible}).  We must have $[F(Y),h]=1$ since $[F(I(g)),h]=1$.  Therefore $1=[g,F(Y)]=[I(g),Y]$ on the nose, or else $[F(I(g)),h]>1$.  In other words, for all $g\neq h$, there exists a unique $Y_{g,h}\in\mathcal{O}(\mathcal{Z}(\mathcal{C}))$ such that
\begin{equation}
F(Y_{g,h})=g\oplus h\oplus b_{g,h}\rho
\end{equation}
for some $b_{g,h}\in\mathbb{Z}_{\geq0}$.  Furthermore, for each $g\in G$, the invertible summands of $F(I(g))$ are all consumed by $F(Y_{g,h})$ for $g\neq h$, except for exactly two copies of $g$.  In general, if $Y\subset I(g)$ is simple and $[g,F(Y)]=b$ for some positive integer $b$, then $[Y,I(g)]=b$, which accounts for $b^2$ copies of $g$ in $F(I(g))$.  As only two copies of $g$ are unaccounted for outside of $Y_{g,h}$ each remaining simple summand of $I(g)$ contains exactly one copy of $g$ under the forgetful functor.  Denote these two simple objects by $Y_g$ and $Y_g'$ such that
\begin{equation}
F(Y_g)=g\oplus b_g\rho,\qquad\text{ and }\qquad F(Y_g')=g\oplus b_g'\rho
\end{equation}
for some $b_g,b_g'\in\mathbb{Z}_{\geq0}$.  We have thus shown each $I(g)$ has exactly $|G|+1$ simple summands, which are pairwise nonisomorphic, i.e.\
\begin{equation}
I(g)\cong Y_g\oplus Y_g'\oplus\bigoplus_{h\neq g}Y_{g,h}.
\end{equation}
In sum, there are $2|G|$ simple objects of $\mathcal{Z}(\mathcal{C})$ of the form $Y_g$ or $Y_g'$ for $g\in G$, and $(1/2)|G|(|G|-1)$ simple objects of the form $Y_{g,h}$ for distinct $g,h\in G$, giving a total of $(1/2)|G|(|G|+3)$ distinct isomorphism classes of simple objects as summands across $I(g)$ for $g\in G$.

\par Lastly, let $J$ be an indexing set for the remaining simple objects $Y_j\in\mathcal{O}(\mathcal{Z}(\mathcal{C}))$, each satisfying $F(Y_j)=b_j\rho$ for some $b_j\in\mathbb{Z}_{\geq1}$ since $[Y_j,I(g)]=0$ for all $g\in G$.  With this labelling,
\begin{equation}
I(\rho)=\left(\bigoplus_{g\in G}(b_g^2Y_g\oplus b_g'^2Y_g'^2)\right)\oplus\left(\bigoplus_{g\neq h}b_{g,h}^2Y_{g,h}\right)\oplus\left(\bigoplus_{j\in J}b_j^2Y_j\right).\label{eq:weak0}
\end{equation}
On the other hand,
\begin{align}
F(I(\rho))=(2|G|+\ell^2)\rho\oplus \ell\bigoplus_{g\in G}g.\label{eq:weak}
\end{align}
Hence the sums of squares of the coefficients $b_g$ and $b_g'$, $b_{g,h}$, and $b_j$, are subject to numerical constraints by Equations (\ref{eq:weak0}) and (\ref{eq:weak}).

\begin{lemma}\label{fortoo}
Let $G$ be a finite abelian group of order $n\in\mathbb{Z}_{\geq1}$, and let $k\in\mathbb{Z}_{\geq1}$.  If $R(G,kn)$ is categorifiable, then
\begin{equation}
\sum_{j\in J}b_j^2\leq\dfrac{1}{2}k^2(n+1)(n-1)+2n.
\end{equation}
\end{lemma}

\begin{proof}
Lemma \ref{lem1} implies $d$ is irrational.  Let $\sigma\in\mathrm{Gal}(\mathbb{Q}(d)/\mathbb{Q})$ be the nontrivial element and recall the Galois action introduced in Section \ref{sec:catstuff}.   Define $\hat{\sigma}(b_g)\in\mathbb{Z}_{\geq0}$ such that $F(\hat{\sigma}(Y_g))=g+\hat{\sigma}(b_g)\rho$ and similarly for $Y_g'$ and $Y_{g,h}$ for all $g,h\in G$ with $g\neq h$.  For all $g\in G$, $b_g+\hat{\sigma}(b_g)=k$, independent of $b_g$ and $b_g'$ by Lemma \ref{gal}, therefore the absolute maximum of $b_g^2+\hat{\sigma}(b_g)^2$ is $k^2$ and the absolute minimum is $k^2/2$.  But we know that $b_e=0$ and $b_e'=k$ as these correspond to formal codegrees $\mathrm{FPdim}(\mathcal{C})$ and $\sigma(\mathrm{FPdim}(\mathcal{C}))$.  The exact same bounds hold for $Y_g'$ over all $g\in G$ since $\{Y_g,Y_g':g\in G\}$ is closed under Galois conjugacy by Lemma \ref{gal}.  Likewise, for all $g\neq h\in G$, the absolute maximum of $b_{g,h}^2+\hat{\sigma}(b_{g,h})^2$ is $(2k)^2=4k^2$ and the absolute minimum is $2k^2$.  But we know that $b_{e,g}=k$ for each $g\neq e$ as these correspond to formal codegrees of the $|G|-1$ one-dimensional representations of $G$.

\par So $\sum_{g\in G}(b_g^2+b_g'^2)+\sum_{g\neq h}b_{g,h}^2$ is minimized by assuming $b_g^2=b_g'^2=k^2/4$ for all $k\neq e$ and $b_{g,h}^2=k^2$ for all $g\neq h\in G$ with $g,h\neq e$.  We compute
\begin{align}
\sum_{g\in G}(b_g^2+b_g'^2)+\sum_{g\neq h}b_{g,h}^2&\geq k^2+(n-1)k^2+2(n-1)\dfrac{1}{4}k^2+k^2\cdot\dfrac{1}{2}(n-1)(n-2) \\
&=\dfrac{1}{2}k^2(n^2+1).
\end{align}
Comparing this to Equation (\ref{eq:weak}), we must have
\begin{align}
\sum_{j\in J}b_j^2&\leq n(2+k^2n)-\dfrac{1}{2}k^2(n^2+1)=\dfrac{1}{2}k^2(n+1)(n-1)+2n.
\end{align}
\end{proof}

%%%%%%%%%%%%%%%%%%%%%%%%%%%%%%%%%%%%%%
%%%%%%%%%%%%%%%%%%%%%%%%%%%%%%%%%%%%%%

%%%%%%%%%%%%%%%%%%%%%
%%%%%%%%%%%%%%%%%%%%%

\subsection{Level bounds for elementary abelian $2$-groups}\label{sec:elem}

Here lies the proof of Theorem \ref{theend}: $R(C_2^m,\ell)$ for $m\in\mathbb{Z}_{\geq1}$ is categorifiable if and only if $\ell=0$, $\ell\in\{1,2\}$ and $m=1$, or $\ell=4$ and $m=2$.  The only external fact that is needed is \cite[Theorem 6.2]{MR3229513}, which states that for $c\in\mathbb{Z}_{\geq1}$ square-free it requires at least $|b|\varphi(2c)$ roots of unity to write $a+b\sqrt{c}$ as a sum of roots of unity.  A stronger version of this fact was needed for the proof that $R(C_3,\ell)$ is categorifiable if and only if $\ell\in\{0,2,3,6\}$ (refer to \cite[Section 4]{MR3229513}).

\begin{proof}{(of Theorem \ref{theend})}
\par The case $G=C_2$ was completed in \cite[Theorem 1.1]{ost15} where it is shown that $R(C_2,\ell)$ is categorifiable if and only if $\ell\in\{0,1,2\}$, and the general case of $\ell=0$ is classical \cite{MR1659954}.  For $G=C_2^m$ with $m>1$, $R(G,\ell)$ is not categorifiable for $0<\ell<|G|$ by \cite[Theorem 1.2]{MR1997336} and Proposition \ref{noncom}, so it is sufficient to consider only $R(G,k|G|)$ with $G$ an elementary abelian $2$-group of order $n\in\mathbb{Z}_{\geq4}$, and $k\in\mathbb{Z}_{\geq1}$.
\par If $R(G,k|G|)$ were categorifiable by a fusion category $\mathcal{C}$, the set of simple objects $\{Y_g,Y_g':g\in G\}$ can be partitioned into $n$ pairs of simple objects whose sums of dimensions are $2+kd$, and the set of simple objects $\{Y_{g,h}:g,h\in G\}$ can be partitioned into $(1/4)n(n-1)$ pairs whose sums of dimensions are $2(2+kd)$.  Indeed, by Lemma \ref{gal}, if $\dim(Y)$ for any simple object $Y$ is $1+bd$ or $2+bd$ where $b\neq k/2$ or $b\neq k$, respectively, then $\dim(Y)+\dim(\hat{\sigma}(Y))$ is $2+kd$ or $2(2+kd)$, respectively, where $\sigma\in\mathrm{Gal}(\mathbb{Q}(d)/\mathbb{Q})$ is the nontrivial element and $Y\not\cong\hat{\sigma}(Y)$.  The remaining simple objects, of even number since $n\geq4$, have dimensions $1+(k/2)d$, or $2+kd$, respectively, and our claim is justified.
\par If the two dimensions in such a pair are $a+rd$ and $a+sd$ for some $a\in\{1,2\}$ such that $r+s=ak$, and therefore $r^2+s^2=a^2k^2-2rs$, their contribution to the sum $\mathrm{Tr}(\theta^2_{I(\rho)})$ is
\begin{equation}\label{zero}
r(a+rd)+s(a+sd)=a(r+s)+(r^2+s^2)d=a^2k+(a^2k^2-2rs)d
\end{equation}
since the squares of the simple objects' twist is 1 by Proposition \ref{theprop}.  Note that $r$ or $s$ may be zero, in which case the corresponding object does not appear as a summand of $I(\rho)$, but this also means it is not contributing to Equation (\ref{zero}).  The sum over all contributions is
\begin{align}
&n(k+k^2d)-\sum_{\substack{r,s \\r+s=k}}(2rs)d+\dfrac{1}{4}n(n-1)(4k+4k^2d)-\sum_{\substack{r,s \\r+s=2k}}(2rs)d \\
=&nkd^2-\sum_{r,s}(2rs)d.
\end{align}
We then have by \cite[Theorem 5.1]{MR2313527},
\begin{align}
\nu_2(\rho)n(2+kd)&=\mathrm{Tr}(\theta^2_{I(\rho)})=nkd^2-\sum_{r,s}(2rs)d+d\sum_{j\in J}b_j^2\theta_j^2
\end{align}
with $\nu_2(\rho)=\pm1$ since $\rho$ is self-dual \cite[Theorem 2.7]{ost15}.  Isolating the sum over the set $J$,
\begin{align}\label{comp}
\sum_{j\in J}b_j^2\theta_j^2&=\sum_{r,s}(2rs)+\nu_2(\rho)\dfrac{n(2+kd)}{d}-nkd \\
&=\sum_{r,s}(2rs)-\dfrac{1}{2}k^2n^2-\left(\dfrac{1}{2}kn-\nu_2(\rho)\right)\sqrt{k^2n^2+4n}
\end{align}
using the facts that $d=(1/2)(kn+\sqrt{k^2n^2+4n})$ and $n(2+kd)/d=\sqrt{k^2n^2+4n}$.  Let $c\in\mathbb{Z}_{\geq1}$ be the largest positive square-free factor of $k^2n^2+4n$.  We must have $c>1$ since $d\not\in\mathbb{Q}$.  From \cite[Theorem 6.2]{MR3229513},
\begin{align}
\sum_{j\in J}b_j^2&\geq\left(\dfrac{1}{2}kn-\nu_2(\rho)\right)\sqrt{k^2n^2+4n}\dfrac{\phi(2c)}{\sqrt{c}}.\label{prev}
\end{align}

\begin{lemma}\label{small}
For all square-free $c\in\mathbb{Z}_{\geq2}$, $\phi(2c)/\sqrt{c}\geq2/\sqrt{3}$.
\end{lemma}

\begin{proof}
If $c>1$ is odd, $\phi(2c)/\sqrt{c}=\phi(c)/\sqrt{c}$ which is minimized when $c$ is prime, since $\phi(pq)/\sqrt{pq}=(\phi(p)/\sqrt{p})(\phi(q)/\sqrt{q})$ and $\phi(p)/\sqrt{p}=\sqrt{p}-1/\sqrt{p}>1$ for all primes $p,q\in\mathbb{Z}_{\geq3}$.  Furthermore, $\sqrt{x}-1/\sqrt{x}$ is a strictly increasing function of $x\in\mathbb{R}_{\geq3}$, hence the smallest value of $\phi(2c)/\sqrt{c}$ when $c>1$ is odd is $\phi(6)/\sqrt{3}=2/\sqrt{3}$.  If $c$ is even, then there exists odd $c'\in\mathbb{Z}_{\geq1}$ such that $\phi(2c)/\sqrt{c}=(\phi(4)/\sqrt{2})\phi(c')/\sqrt{c'}=\sqrt{2}\phi(c')/\sqrt{c'}$, whose smallest value is $\sqrt{2}>2/\sqrt{3}$ when $c'=1$.
\end{proof}

Combining Lemma \ref{small} with the bound in Lemma \ref{fortoo} and the inequality in (\ref{prev}), our general constraint is
\begin{equation}\label{only}
\dfrac{1}{2}k^2(n+1)(n-1)+2n\geq\dfrac{2}{\sqrt{3}}\left(\dfrac{1}{2}kn-\nu_2(\rho)\right)\sqrt{k^2n^2+4n},
\end{equation}
which is strong enough to reduce our argument to a finite number of cases.  Assume first that $\nu_2(\rho)=1$.  Squaring both sides of (\ref{only}), and subtracting the right-hand side from the left, gives the constraint that
\begin{equation}\label{bruv}
(-n^4-6n^2+3)k^4+(16n^3)k^3+(8n^3-16n^2-24n)k^2+(64n^2)k+48n^2-64n
\end{equation}
must be nonnegative (after we have scaled by $12$).  Denote the expression in (\ref{bruv}) by $f(n,k)$ and consider it as a quartic function of $k\in\mathbb{R}$ for each $n\in\mathbb{Z}_{\geq4}$.  We compute
\begin{equation}
f(n,5/\sqrt{n})=-377n^2+2320n^{3/2}-464 n-4350+\dfrac{1875}{n^2}
\end{equation}
which is negative for $n\geq64$.  So it suffices to prove that $df/dk<0$ for all $k>5/\sqrt{n}$ to prove there are no $k\in\mathbb{Z}_{\geq1}$ satisfying (\ref{only}) when $n\geq32$, since $5/\sqrt{64}<1$.  We compute
\begin{equation}
\dfrac{1}{4}\dfrac{df}{dk}=(-n^4-6n^2+3)k^3+(12n^3)k^2+(4n^3-8n^2-12n)k+16n^2.
\end{equation}
As $df/dk$ is a cubic function in $k$ with a negative leading coefficient, the three facts that $(df/dk)|_{k=5/\sqrt{n}}<0$, $(d^2f/dk^2)|_{k=5/\sqrt{n}}<0$, and $(d^3d/dk^3)|_{k=5/\sqrt{n}}<0$ for $n\geq64$ suffice to prove that $df/dk<0$ for all $k>5/\sqrt{n}$.  These are functions of a single variable $n\geq64$ so we suppress this verification for brevity.  Lastly we must check the cases $n\in\{4,8,16,32\}$ individually.  We compute
\begin{align}
f(4,k)&=-349 k^4 + 1024 k^3 + 160 k^2 + 1024 k + 512, \\
f(8,k)&=-4477 k^4 + 8192 k^3 + 2880 k^2 + 4096 k + 2560, \\
f(16,k)&=-67069 k^4 + 65536 k^3 + 28288 k^2 + 16384 k + 11264,\text{ and} \\
f(32,k)&=-1054717 k^4 + 524288 k^3 + 244992 k^2 + 65536 k + 47104.
\end{align}
We suppress the elementary analytic arguments that $f(4,k)<0$ for integers $k>3$, $f(8,k)<0$ for integers $k>2$, $f(16,k)<0$ for integers $k>1$, and $f(32,k)<0$ for integers $k>0$.  Now that the set of remaining $n,k$ is finite, the remaining $n,k$ (and thus square-free $c$) such that $f(n,k)\geq0$ must satisfy the stronger, or rather, more specific constraint
\begin{equation}\label{endgame}
\dfrac{1}{2}k^2(n^2-1)+2n\geq\dfrac{\phi(2c)}{\sqrt{c}}\left(\dfrac{1}{2}kn-1\right)\sqrt{k^2n^2+4n}.
\end{equation}
There is nothing left to prove for $n=32$, but when $n=16$ and $k=1$, then $c=5$.  This triple $n,k,c$ does not satisfy (\ref{endgame}).  When $n=8$ and $k=1$, $c=6$, and when $n=8$ and $k=2$, $c=2$.  Neither triple $n,k,c$ satisfies (\ref{endgame}).  Lastly, when $n=4$ and $k=3$, $c=10$; when $n=4$ and $k=2$, $c=5$; when $n=4$ and $k=1$, $c=2$.  The first two triples do not satisfy (\ref{endgame}), but the last triple does. This is expected as $R(C_2^2,4)$ is known to be categorifiable.
\par To finish the argument, simply note that the inequality in (\ref{only}) when $\nu_2(\rho)=-1$ is a strictly stronger constraint than when $\nu_2(\rho)=1$, and so our argument is complete. 
\end{proof}

%%%%%%%%%%%%%%%%%%%%%%%%%
%%%%%%%%%%%%%%%%%%%%%%%%%%%
%%%%%%%%%%%%%%%%%%%%%%%%%%

\subsection{Level restrictions for primes $p\equiv3\pmod{4}$}\label{sec:prime}

\par Let $p$ be prime such that $p\equiv3\pmod{4}$ and $k\in\mathbb{Z}_{\geq1}$.  We assume henceforth that $\mathcal{C}$ is a categorification of the near-group fusion ring $R(C_p,kp)$.  Denote $G:=C_p=\{g,g^2,\ldots,g^p=e\}$.  Then \cite[Theorem 2.5]{ost15} implies
\begin{equation}
0=\mathrm{Tr}(\theta_{I(g)})=(1+b_gd)\theta_g+(1+b_g'd)\theta_g'+\sum_{j=2}^p(2+b_{g,g^j}d)\theta_{g,g^j}
\end{equation}
where $d$ is the largest root of $x^2-kpx-p$.  Therefore
\begin{equation}\label{continue}
\left(b_g\theta_g+b_g'\theta_g'+\sum_{j=2}^pb_{g,g^j}\theta_{g,g^j}\right)(-d)=\theta_g+\theta_g'+2\sum_{j=2}^p\theta_{g,g^j}.
\end{equation}
Dividing by the sum on the left-hand side would imply $d\in\mathbb{Q}(\zeta_p)$ by Lemma \ref{theprop}, which is not true when $d$ is irrational since $\mathbb{Q}(\zeta_p)$ contains no real quadratic field, hence both sides of (\ref{continue}) are zero.  Firstly, this implies for all $e\neq g\in G$,
\begin{equation}
0=\theta_g+\theta_g'+2\sum_{j=2}^p\theta_{g,g^j},
\end{equation}
which is a sum of $2p$ roots of unity of order $p$, which is zero if and only if every root appears in the sum exactly twice.  Therefore for each nontrivial $g\in G$, $\theta_g=\theta_g'$ are the same primitive $p$th root of unity since the root 1 appears as $\theta_{e,g}$ for all $g\neq e$.  Secondly,
\begin{equation}
0=b_g\theta_g+b_g'\theta_g'+\sum_{j=2}^pb_{g,g^j}\theta_{g,g^j}=(b_g+b_g')\theta_g+\sum_{j=2}^pb_{g,g^j}\theta_g^{\iota(j)}
\end{equation}
for some reindexing $\iota:\{2,\ldots,p\}\to\{2,\ldots,p\}$ which depends on $\theta_g$.  Therefore $b_{g,g^j}=b_g+b_g'$ for all $2\leq j\leq p$.  In other words, $\mathrm{FPdim}(Y_{g,g^i})=\mathrm{FPdim}(Y_{g,g^j})$ for all $2\leq i,j\leq p$.
\par Now assume $b_{g,g^j}\neq k$ for some $2\leq j\leq p$.  Then with $\sigma\in\mathrm{Gal}(\mathbb{Q}(d)/\mathbb{Q})$ the nontrivial automorphism, Lemma \ref{gal} states that $\mathrm{FPdim}(\hat{\sigma}(Y_{g,g^j}))=2+(2k-b_{g,g^j})d$.  Since $2k-b_{g,g^j}\neq b_{g,g^j}$ by assumption, then $\hat{\sigma}(Y_{g,g^j})\not\cong Y_{g,g^k}$ for any $k$.  Therefore $\hat{\sigma}(Y_{g,g^j})\cong Y_{g^r,g^s}$ with $g^r,g^s\neq g$.  But for instance $Y_{g,g^s}\subset I(X_{g^s})$, contradicting the fact that $\mathrm{FPdim}(Y_{g^t,g^s})=\mathrm{FPdim}(Y_{g^u,g^s})$ for all $t,u\neq s$.  We must then conclude $b_{g^j,g^\ell}=k$ for all $1\leq j,\ell\leq p$ with $j\neq\ell$.  

\begin{lemma}
Let $p\equiv3\pmod{4}$ be prime and $k\in\mathbb{Z}_{\geq0}$.  If the near-group fusion ring $R(C_p,kp)$ is categorifiable, then there exist $0\leq r,s\leq k$ such that $r+s=k$,
\begin{align}
\sum_{j\in J}b_j^2\theta_j+\sum_{j\in J}b_j^2\theta^{-1}_j&=-4rs-k\sqrt{k^2p^2+4p},\text{ and}\label{ichi} \\
\sum_{j\in J}b_j^2\theta^2_j+\sum_{j\in J}b_j^2\theta^{-2}_j&=-4rs-(k-2\nu_2(\rho))\sqrt{k^2p^2+4p}.\label{ni}
\end{align}
\end{lemma}

\begin{proof}
It is well-known that $\mathrm{FPdim}(\mathcal{C})=\tau^+(\mathcal{Z}(\mathcal{C}))$ \cite[Theorem 1.2]{mug3} when $\mathcal{C}$ is pseudounitary, where $\tau^\pm$ are the positive and negative Gauss sums \cite[Section 8.15]{tcat}, and therefore $2\mathrm{FPdim}(\mathcal{C})=\tau^+(\mathcal{Z}(\mathcal{C}))+\tau^-(\mathcal{Z}(\mathcal{C}))$ since $\dim(\mathcal{C})$ is real.  By definition,
\begin{align}\label{eqdef}
\rho^+(\mathcal{Z}(\mathcal{C}))+\rho^-(\mathcal{Z}(\mathcal{C}))&=\sum_{Y\in\mathcal{O}(\mathcal{Z}(\mathcal{C}))}\dim(Y)^2(\theta_Y+\theta_Y^{-1}).
\end{align}
Recall that $\theta_g$, $\theta_g'$, and $\theta_{g,h}$ are primitive $p$th roots of unity unless $g$ or $h$ is the identity.  Therefore, we may collect these simple objects into sets of $(p-1)/2$ by square Galois conjugacy.  Let $\sigma\in\mathrm{Gal}(\mathbb{Q}(\zeta_N)/\mathbb{Q})$, where $N\in\mathbb{Z}_{\geq1}$ is the conductor of $\mathcal{C}$ be any element such that $\sigma(\zeta_p)=\zeta_p^2$, i.e.\ $\sigma|_{\mathbb{Q}(\zeta_p)}$ is a generator of $\mathrm{Gal}(\mathbb{Q}(\zeta_p)/\mathbb{Q})$.  Such an element exists by \cite[Theorem 3.9]{paul} since $p\mid\dim(\mathcal{C})$.  We know that for all $Y\in\mathcal{O}(\mathcal{C})$, $\theta_{\hat{\sigma}^2(Y)}=\sigma^2(\theta_Y)$ \cite[Theorem II(iii)]{dong2015congruence}, so if $\theta_Y$ is a primitive $p$th root of unity, $\hat{\sigma}^{2j}(Y)$ are distinct simple objects for $1\leq j\leq(p-1)/2$.  Now $\sigma|_{\mathbb{Q}(d)}$ is either trivial or order 2, hence $\mathrm{FPdim}(\hat{\sigma}^{2j}(Y))^2=\mathrm{FPdim}(Y)^2$ as well (refer to Equation (\ref{twentyseven})).  In particular, 
\begin{equation}
\alpha:=\sum_{j=1}^{(p-1)/2}\mathrm{FPdim}(\hat{\sigma}^{2j}(Y))^2\theta_{\hat{\sigma}^2(Y)}=\dfrac{1}{2}(-1\pm i\sqrt{p})\mathrm{FPdim}(Y)^2.
\end{equation}
Therefore $\alpha+\overline{\alpha}=-\mathrm{FPdim}(Y)^2$.  We may collect the $Y_{g,h}$ for nontrivial $g,h\in G$ into
\begin{equation}
\dfrac{(1/2)p(p-1)-(p-1)}{(1/2)(p-1)}=p-2
\end{equation}
distinct collections of this type, contributing $-(p-2)(2+kd)^2$ to the sum in Equation (\ref{eqdef}).  The restrictions that $\theta_g=\theta_g'$ for all $g\in G$ and $b_g+b_g'=k$ forces a similar partition by square Galois conjugacy of $\{Y_g,Y_g':g\neq e\}$ into two sets, each consisting of simple objects of dimensions $1+rd$ and $1+sd$ for some unknown $0\leq r,s\leq k$.  These contribute $-2(1+rd)^2$ and $-2(1+sd)^2$ to  the sum in Equation (\ref{eqdef}), respectively.  The summands of $I(\mathbbm{1}_\mathcal{C})$ have trivial twist, so together they contribute $2(1+(1+kd)^2+(p-1)(2+kd)^2)$ to the sum in Equation (\ref{eqdef}).  To isolate the unknown summands on the right-hand side of Equation (\ref{eqdef}) corresponding to the set $J$, we compute
\begin{align}
&2p(2+kd)-(2(1+(1+kd)^2+(p-1)(2+kd)^2) \\
&-(p-2)(2+kd)^2-2(1+rd)^2-2(1+sd)^2))=d^2(k^2p-2kd-4rs).
\end{align}
Therefore
\begin{align}
\sum_{j\in J}b_j^2\theta_j+\sum_{j\in J}b_j^2\theta^{-1}_j&=k^2p-2kd-4rs=-4rs-k\sqrt{k^2p^2+4p}.
\end{align}
The second equality follows from identical reasoning applied to the fact that $\mathrm{Tr}(\theta_{I(\rho)}^2)+\overline{\mathrm{Tr}(\theta_{I(\rho)}^2)}=2\nu_2(\rho)\mathrm{FPdim}(\mathcal{C})$ by \cite[Theorem 5.1]{MR2313527}.
\end{proof}

\begin{proposition}\label{even}
Let $p\equiv3\pmod{4}$ be prime and $k\in\mathbb{Z}_{\geq0}$.  If the near-group fusion ring $R(C_p,kp)$ is categorifiable and $k\neq1$, then $k$ is even.
\end{proposition}

\begin{proof}
The norm of $\mathrm{FPdim}(\mathcal{C})$ is $p^3(k^2p+4)$, so the conductor $N\in\mathbb{Z}_{\geq1}$ of $\mathcal{Z}(\mathcal{C})$ is even if and only if $k$ is even \cite[Theorem 3.9]{paul}.  If $k$ is not even, then $\zeta_N\stackrel{\sigma}{\mapsto}\zeta_N^2$ is a Galois automorphism of $\mathbb{Q}(\zeta_N)$.  Therefore when $k>0$, the right-hand sides of Equations (\ref{ichi}) and (\ref{ni}) are conjugate.  This could only occur if $\sigma(\sqrt{k^2p^2+4p})=-\sqrt{k^2p^2+4p}$ which implies $-k=k\pm2$, i.e. $k=1$.
\end{proof}

\begin{proposition}\label{xbound}
Let $p\equiv3\pmod{4}$ be prime, $m\in\mathbb{Z}_{\geq1}$ and $x\in\mathbb{Z}_{\geq1}$ be the square-free part of $m^2p+1$.  If the near-group fusion ring $R(C_p,2mp)$ is categorifiable, then
\begin{equation}
\dfrac{\phi(x)}{\sqrt{x}}\leq\left(\dfrac{p+1}{p-1}\right)\sqrt{p+1/m^2}.
\end{equation}
\end{proposition}

\begin{proof}
Note that if $k\neq1$, then $k=2m$ is even by Proposition \ref{even}.  Hence $k^2p^2+4p=4p(m^2p+1)$.  Since $m^2p+1\equiv1\pmod{p}$, then $p\nmid x$ but $p\mid c$.  Our main constraint follows from \cite[Theorem 6.2]{MR3229513} along with Equation (\ref{ichi}), which imply that
\begin{equation}
\sum_{j\in J}b_j^2\geq k(p-1)\sqrt{m^2p+1}\dfrac{\phi(x)}{\sqrt{x}}.\label{den}
\end{equation}
But we compute
\begin{align}
\sum_{j\in J}\mathrm{FPdim}(Y_j)^2=\sum_{j\in J}b_j^2d^2&\leq\mathrm{FPdim}(\mathcal{Z}(\mathcal{C}))-\sum_{g\neq h}\mathrm{FPdim}(Y_{g,h})^2 \\
%&=\dfrac{1}{2}p(p+1)(2+kd)^2 \\
&=2p(p+1)(m^2p+1)(1+kd).\label{num}
\end{align}
Therefore, using (\ref{num}) to bound the numerator, and (\ref{den}) to bound the denominator,
\begin{align}
&&1=\dfrac{\sum_{j\in J}b_j^2d^2}{\sum_{j\in J}b_j^2d^2}&\leq\dfrac{2p(p+1)(m^2p+1)(1+kd)}{2d^2k(p-1)\sqrt{m^2p+1}}\cdot\dfrac{\sqrt{x}}{\phi(x)} \\
&&&=\dfrac{(p+1)}{(p-1)}\cdot\dfrac{\sqrt{k^2p+4}}{k}\cdot\dfrac{\sqrt{x}}{\phi(x)} \\
\Rightarrow&&\dfrac{\phi(x)}{\sqrt{x}}&\leq\left(\dfrac{p+1}{p-1}\right)\cdot\dfrac{\sqrt{m^2p+1}}{m}.
\end{align}
\end{proof}

\begin{example}
For any fixed prime $p\in\mathbb{Z}_{\geq2}$, the bound of Proposition \ref{xbound} is crippling.  For instance if $p=7$, then we require $\phi(x)/\sqrt{x}\leq8\sqrt{2}/3$ when $m\geq1$.  It is easy to verify that $2,3,5,13\nmid x$, e.g.\ $k^2\equiv0,1,4\pmod{5}$, thus $7k^2+4\equiv4,1,2\pmod{5}$. Since $\phi(17)>8\sqrt{2}/3$, we have $x=1$ or $x=11$.  Although this does not immediately bound the level $k$, one can check which small levels attain these $x$-values.  We see that the smallest for $x=1$ are $k=6,96,1530,24384,388614,\ldots$, and the smallest with $x=11$ are $k=1,10,7030,\ldots$  Moreover, these eight are the only possible levels with $k<2\cdot10^6$.
\end{example}

\begin{conjecture}
Let $p\in\mathbb{Z}_{\geq2}$ be prime and $\ell\in\mathbb{Z}_{\geq0}$.  If the near-group fusion ring $R(C_p,\ell)$ is categorifiable, then $\ell<p^2$.
\end{conjecture}

\bibliographystyle{plain} 
\bibliography{bib}

\end{document}